\documentclass[12pt,reqno]{amsart}

\usepackage{amssymb} 
\usepackage{mathtools}
\usepackage{mathrsfs} 
\usepackage{bm} 
\usepackage[hidelinks]{hyperref} 
\usepackage{esint} 
\usepackage{appendix}
\usepackage{cite} 
\usepackage{cleveref} 

\usepackage [english]{babel}
\usepackage [autostyle, english = american]{csquotes}
\MakeOuterQuote{"}

\calclayout

\numberwithin{equation}{section}
\newtheorem{theorem}{Theorem}[section]
\newtheorem{corollary}[theorem]{Corollary}
\newtheorem{lemma}[theorem]{Lemma}

\theoremstyle{definition}
\newtheorem{definition}{Definition}[section]
\newtheorem{remark}{Remark}[section]
\newtheorem{assumption}{Assumption}[section]

\newcommand\R{\mathbb R}
\newcommand\divg{\mathop{\mbox{\rm div}}}
\newcommand{\lagr}[1]{\boldsymbol{\mathbf{#1}}}
\newcommand\adj{\mathop{\mbox{\rm adj}}}

\title[Parabolic Lam\'{e} system with rough coefficients]{Gaussian bounds of fundamental matrix and maximal $L^1$ regularity for Lam\'{e} system with rough coefficients}

\author{Huan XU}

\subjclass{35B65; 35K08; 47D06.}
\keywords{Heat kernel, Maximal regularity, Abstract Cauchy problem, Lam\'{e} system, Pressureless flow.}

\AtEndDocument{%
  \par
  \medskip
  \begin{tabular}{@{}l@{}}%
    \textsc{Department of Mathematics, Auburn University, Auburn, AL 36849, USA}\\
    \textit{E-mail address}: \texttt{hzx0016@auburn.edu}
  \end{tabular}}

\begin{document}

\maketitle

\begin{abstract}
The purpose of this paper is twofold. First, we use a classical method to establish Gaussian bounds of the fundamental matrix of a generalized parabolic Lam\'{e} system with only bounded and measurable coefficients. Second, we derive a maximal $L^1$ regularity result for the abstract Cauchy problem associated with a composite operator. In a concrete example, we also obtain maximal $L^1$ regularity for the Lam\'{e} system, from which it follows that the Lipschitz seminorm of the solutions to the Lam\'{e} system is globally $L^1$-in-time integrable. As an application, we use a Lagrangian approach to prove a global-in-time well-posedness result for a viscous pressureless flow provided that the initial velocity satisfies a scaling-invariant smallness condition. The method established in this paper might be a powerful tool for studying many issues arising from viscous fluids with truly variable densities.
\end{abstract}

\bigskip


\bigskip

\section{Introduction}

\bigskip

This work is motivated by the study of the global well-posedness of the Cauchy problem for a class of hyperbolic-parabolic coupled systems modeling the motion of fluids. Probably the most famous example is the system of Navier-Stokes equations (see \cite{Lions book I,Lions book II}). In a fluid flow, the law of conservation of mass can be formulated mathematically using the continuity equation, given in differential form as
\begin{align}\label{continuity equation}
\partial_t\rho+\divg(\rho u)=0,\ \ \mathrm{in}\ (0,\infty)\times\R^n,
\end{align}
where $\rho$ is the density (mass per unit volume) and $u$ is the flow velocity field. The law of conservation applied to momentum gives the momentum equation of the form
\begin{align}\label{momentum equation with pressure}
\rho(\partial_tu+u\cdot\nabla u)-\mathcal{A}u+\nabla P=0,\ \ \mathrm{in}\ (0,\infty)\times\R^n,
\end{align}
where $P$ is a scalar pressure and $\mathcal{A}$ is a dissipative operator. The fluid flow can be either incompressible or compressible. However, in the compressible case, we will only consider pressureless flows, which means we drop the pressure term in \eqref{momentum equation with pressure} and write
\begin{align}\label{pressureless momentum equation}
\rho(\partial_tu+u\cdot\nabla u)-\mathcal{A}u=0,\ \ \mathrm{in}\ (0,\infty)\times\R^n.
\end{align}
Note that \eqref{continuity equation}$\&$\eqref{pressureless momentum equation} with $\mathcal{A}=\Delta$ can be viewed as a viscous regularization for the model of inviscid pressureless gases.

The system of the form \eqref{continuity equation}$\&$\eqref{pressureless momentum equation} has been studied by several authors. When $n=1$ and $\mathcal{A}=\Delta$, Boudin \cite{boudin siam 2000} proved the existence of a global smooth solution to \eqref{continuity equation}$\&$\eqref{pressureless momentum equation}. Perepelitsa \cite{Perepelitsa 2010} considered \eqref{continuity equation}$\&$\eqref{pressureless momentum equation} as a simplified model of compressible isentropic Navier-Stokes equations and he proved the global existence of a small energy weak solution with the density being a nonnegative bounded function throughout the half-space $\R_+^3$. Recently, Danchin et al. \cite{danchin jee 2021} formally derived the system \eqref{continuity equation}$\&$\eqref{pressureless momentum equation}, with $\mathcal{A}$ being the Laplacian or the Lam\'{e} operator, as a model of some collective behavior phenomena. They also proved the existence and uniqueness of a global solution with the initial density being only bounded and close to a constant in $L^\infty$-norm.

In this work, we are particularly interested in solving \eqref{continuity equation}$\&$\eqref{momentum equation with pressure} and \eqref{continuity equation}$\&$\eqref{pressureless momentum equation} via the Lagrangian method (see \cite{danchin aif 2014,danchin cpam 2012}). The advantage is that one can convert the hyperbolic-parabolic coupled system into a parabolic system. Then the uniqueness and stability issues can be tackled in a relatively easy way compared to solving the system in Eulerian coordinates. In this framework, under a scaling-invariant smallness condition on the initial velocity, the heart of the matter is to bound the quantity $\int_0^\infty\|\nabla u(t)\|_\infty\,dt$ because this would imply the existence of global-in-time coordinate transformations. Without going into details, we are led to consider the linearized system of the Lagrangian formulation of \eqref{continuity equation}$\&$\eqref{pressureless momentum equation} (or, \eqref{continuity equation}$\&$\eqref{momentum equation with pressure}) that reads
\begin{align}\label{linearized Lagrangian formulation}
\rho(x)\partial_tu-\mathcal{A}u=f.
\end{align}
Note that the coefficient $\rho$ is now a time-independent function. In the incompressible case, the operator $\mathcal{A}$ in \eqref{linearized Lagrangian formulation} is different from the one in \eqref{momentum equation with pressure}. Indeed, we need to introduce a so-called Stokes operator to unify the internal force (viscosity and pressure) in \eqref{momentum equation with pressure} (see \cite{huan 2021}). But $\mathcal{A}$ is just what it used to be in the compressible pressureless case. {\it Now the main purpose of this paper is to derive the estimate of $\int_0^\infty\|\nabla u(t)\|_\infty\,dt$ for solutions $u$ to \eqref{linearized Lagrangian formulation} under least regularity assumption on the coefficient $\rho$.} To achieve this, we shall study the maximal $L^1$-in-time regularity for solutions to \eqref{linearized Lagrangian formulation} in homogeneous type spaces.

If $\rho$ is close to some constant, \eqref{linearized Lagrangian formulation} is essentially a perturbation of a linear system with constant coefficients. In this case, Danchin and Mucha \cite{danchin cpam 2012} established a maximal $L^1$ regularity result for a linear Stokes system with discontinuous coefficients (including piecewise constant densities). In our work, we do not assume any smallness condition on the fluctuation of the density. So \eqref{linearized Lagrangian formulation} is no longer a perturbation problem, but we can rewrite it as the following abstract Cauchy problem 
\begin{align}\label{transform linearized Lagrangian formulation}
\partial_tu-\rho^{-1}\mathcal{A}u=\Tilde{f}\vcentcolon=\rho^{-1}f
\end{align}
associated with the composite operator $\rho^{-1}\mathcal{A}$, provided that $\rho$ has a positive lower bound. What makes the maximal $L^1$ regularity for \eqref{transform linearized Lagrangian formulation} possible is the observation that the composite operator $\rho^{-1}\mathcal{A}$ behaves similarly to some operator with constant coefficients, in the sense that certain Besov-type norms defined via the semigroups generated by both operators are equivalent. This was one of the key observations made in \cite{huan 2021} in which the author of the present paper proved the first maximal $L^1$ regularity result concerning viscous incompressible fluids with truly variable densities. 

In \cite{huan 2021}, we were only able to work in the $L^2$ (in space) framework due to the presence of pressure. In this paper, we mainly focus on the pressureless case, and we will work in the general $L^p$ (in space) framework. A practical benefit of doing so is that one can lower the regularity of the density (see \cite{danchin cpam 2012}). For the analysis in \cite{huan 2021} to adapt to the $L^p$ framework, we need to make the extra effort to obtain pointwise bounds for the kernel of the semigroup generated by $\rho^{-1}\mathcal{A}$. Let us consider two concrete examples. For $\mathcal{A}=\Delta$ (the Laplacian), McIntosh and Nahmod \cite{mcintosh 2000} proved that the kernel of the $L^2$ semigroup $e^{t\rho^{-1}\Delta}$ generated by $\rho^{-1}\Delta$ satisfies Gaussian bounds (see also \cite{duong DIE 1999}). This guarantees that the semigroup $e^{t\rho^{-1}\Delta}$ extrapolates to a bounded analytic semigroup on $L^p$, $1<p<\infty$. Note that the kernel of $e^{t\rho^{-1}\Delta}$ is essentially a scalar kernel. If $\mathcal{A}$ is the Lam\'{e} operator (see Section \ref{preliminary section}), however, \eqref{linearized Lagrangian formulation} is a truly coupled system whose fundamental matrix does not necessarily satisfy Gaussian bounds. Nevertheless, we can prove the bounds for the fundamental matrix and its derivatives using a rather classical method if the dimensions of the Euclidean space $\le3$. The tricks are due to Davies, one is to use Sobolev inequalities to bound $L^\infty$-norm (see \cite{davies jfa 1995}), the other is a perturbation technique to obtain exponential decay (see \cite{davies ajm 1987}). In the spirit of \cite{mcintosh 2000,duong DIE 1999}, once we obtain Gaussian upper bounds of the fundamental matrix (denoted by $K_t(x,y)$), we can easily get the $C^{1,\gamma}$ estimates for the kernel $K_t(x,y)\rho^{-1}(y)$.

Before we study the maximal regularity for \eqref{transform linearized Lagrangian formulation}, we will establish a maximal $L^1$-in-time regularity result for the abstract Cauchy problem
\begin{eqnarray}\label{ACP associated with composite operator}
\left\{\begin{aligned}
&u'(t)-\mathcal{S}u(t)=f(t),\\
&u(0)=x
\end{aligned}\right.
\end{eqnarray}
in homogeneous type spaces. Let us assume that $\mathcal{S}$ is an unbounded linear operator on a Banach space $(X,\|\cdot\|)$ that generates a bounded analytic semigroup $e^{t\mathcal{S}}$. Given \eqref{ACP associated with composite operator} with $x=0$, $\mathcal{S}$ is said to have maximal $L^r$-in-time regularity in $X$ for $r\in[1,\infty]$, if for every $f\in L^r((0,\infty);X)$, \eqref{ACP associated with composite operator} has a unique solution verifying
\begin{align}\label{definition of maximal regularity}
\|\mathcal{S}u\|_{L^r((0,\infty);X)}\le C\|f\|_{L^r((0,\infty);X)}.
\end{align}
The maximal $L^r$ regularity issue for $r\in(1,\infty)$ has been extensively studied in the literature. We refer to \cite{da prato jmpa 1975,denk et al book 2003,dore and venni 1987,giga jfa 1991,pruss conf 2002}, amongst which \cite{da prato jmpa 1975} also covered the $L^1$ theory, but the global-in-time estimate \eqref{definition of maximal regularity} holds only if $0$ belongs to the resolvent set $\rho(\mathcal{S})$ of $\mathcal{S}$ (i.e., $\mathcal{S}^{-1}\in\mathscr{L}(X)$). It goes without saying that such a condition is very demanding in many concrete examples. Recently, Ri and Farwig \cite{ri arxiv 2020} established maximal $L^1$ regularity for $\mathcal{S}$ in inhomogeneous type spaces without assuming $0\in\rho({\mathcal{S}})$. Later, a similar result in the homogeneous space setting was proved by Danchin et al. \cite{danchin arxiv 2020}. The authors in \cite{danchin arxiv 2020} also nicely explained the importance of maximal $L^1$ regularity for parabolic systems in homogeneous spaces. Our work is more relevant to the one in \cite{danchin arxiv 2020}. But \cite{danchin arxiv 2020} did not cover maximal regularity in homogeneous spaces with negative regularity. For us, working in spaces with negative regularity can weaken the regularity of the density. Our method is motivated by our prior work \cite{huan 2021} in which we obtained maximal $L^1$ regularity for a generalized Stokes operator with variable coefficients in homogeneous Besov spaces. It turns out that the strategy of the proof of the concrete result in \cite{huan 2021} works equally well for the abstract problem.

This paper is organized as follows. Section \ref{preliminary section} is a short review of some basics needed in this paper. In Section \ref{heat kernel estimate section}, we prove the $C^{1,\gamma}$ regularity for $K_t(x,y)b(y)$, where $K_t(x,y)$ is the matrix-valued heat kernel of $-b\mathcal{L}$ and $\mathcal{L}$ is the Lam\'{e} operator. We remark that the coefficient $b$ is only bounded and bounded from below by a positive constant. In Section \ref{maximal regularity for ACP section}, we follow our prior work \cite{huan 2021} closely and derive the maximal $L^1$ regularity for the abstract Cauchy problem \eqref{ACP associated with composite operator} when $\mathcal{S}$ is a composition of bounded and unbounded operators. Then, in Section \ref{concrete parabolic system section}, we apply the abstract theory to study the maximal $L^1$ regularity for \eqref{linearized Lagrangian formulation}, where $\mathcal{A}$ is the Laplacian or the Lam\'{e} operator. Section \ref{compressible pressureless flow section} is devoted to the global-in-time well-posedness of the pressureless system \eqref{continuity equation}$\&$\eqref{pressureless momentum equation}.

\bigskip

\noindent{\bf Notations:} Throughout, the letter $C$ denotes a harmless positive constant that may change from line to line, but whose meaning is clear from the context. The notation $a\lesssim b$ means $a\leq Cb$ for some $C$, and $a\simeq b$ means $a\lesssim b$ and $b\lesssim a$. For two quantities $a, b$, we denote by $a\vee b$ the bigger quantity and by $a\wedge b$ the smaller one. We denote by $\|\cdot\|_p$ the Lebesgue $L^p$-norm. For a matrix $A$, $A^{\intercal}$ denotes its transpose. For a Banach space $X$, $\mathscr{L}(X)$ denotes the space of all continuous linear operators on $X$. For $q\in[1,\infty]$, we may write $\|\cdot\|_{L_t^q(X)}$ for the norm of the space $L^q((0,t);X)$, and $\|\cdot\|_{L^q(X)}$ for the norm of $L^q(\R_+;X)$, where $\R_+=(0,\infty)$. Finally, we denote operators on Banach spaces by "mathcal" letters (e.g., $\mathcal{A}$, $\mathcal{B}$, $\mathcal{S}$, etc.).

\bigskip

\section{Preliminaries}\label{preliminary section}

\bigskip

\subsection{Semigroups and abstract Cauchy problem}
In this paper, we only consider real vector spaces and no complexification is needed. Let $(X,\|\cdot\|)$ be a real Banach space. We adopt the concept that a $C_0$ semigroup $\{\mathcal{T}(t)\}_{t\ge0}$ on $X$ is called a bounded $C_0$ semigroup if $\|\mathcal{T}(t)\|\le M<\infty$ for all $t\ge0$, while it is called a contraction semigroup if $M=1$. 

\begin{definition}
$\{\mathcal{T}(t)\}_{t\ge0}$ is called a bounded analytic semigroup on $X$ if it is a bounded $C_0$ semigroup with generator $\mathcal{A}$ such that $\mathcal{T}(t)x\in D(\mathcal{A})$ for all $x\in X$ and $t>0$, and
\begin{align}\label{real analyticity criterion}
\sup_{t>0}\|t\mathcal{A}\mathcal{T}(t)x\|\le C\|x\|,\ \ \forall x\in X.
\end{align}
\end{definition}

\begin{remark}
In applications, one only needs to show \eqref{real analyticity criterion} for $x$ belonging to a dense subspace of $X$ since $\mathcal{A}$ is closed.
\end{remark}

\begin{remark}
In fact, \eqref{real analyticity criterion} is also a real characterization of complex analyticity, see, for example, {\cite[Theorem~4.6]{engel book 2000}}, or {\cite[Theorem~3.7.19]{arendt book 2011}}.
\end{remark}

Let $(H,\langle\cdot,\cdot\rangle)$ be a real Hilbert space. A linear operator $\mathcal{A}:D(\mathcal{A})\subset H\rightarrow H$ is called dissipative on $H$ if and only if
\begin{align*}
\langle  \mathcal{A}x,x\rangle\le0,\ \forall x\in D(\mathcal{A}).
\end{align*}
We have the following well-known result:
\begin{theorem}\label{generation theorem of selfadjoint operator}
Let $\mathcal{A}$ be a self-adjoint operator on $H$. Then $\mathcal{A}$ generates an analytic semigroup of contraction $\{e^{t\mathcal{A}}\}_{t\ge0}$ if and only if $\mathcal{A}$ is dissipative. Moreover, $e^{t\mathcal{A}}$ is self-adjoint on $H$ for every $t\ge0$.
\end{theorem}
For the complex version of Theorem \ref{generation theorem of selfadjoint operator}, we refer to {\cite[Example~3.7.5]{arendt book 2011}} and {\cite[Corollary~3.3.9]{arendt book 2011}}. 

In applications, we will first apply Theorem \ref{generation theorem of selfadjoint operator} to construct a semigroup on $L^2$, and then extrapolate it to some Besov spaces. However, it is usually not easy to identify the generator of the new semigroup. In this situation, we wish
to identify the generator restricted on a dense subspace of its domain. Recall that a subspace $Y$ of the domain $D(\mathcal{A})$ of a linear operator $\mathcal{A}:D(\mathcal{A})\subset X\rightarrow X$ is called a core for $\mathcal{A}$ if $Y$ is dense in $D(\mathcal{A})$ for the graph norm $\|x\|_{D(\mathcal{A})}\vcentcolon=\|x\|+\|\mathcal{A}x\|$. In other words, $Y$ is a core for $\mathcal{A}$ if and only if $\mathcal{A}$ is the closure of $\mathcal{A}|_{Y}$. The next result gives a useful sufficient condition for a subspace to be a core for the generator.

\begin{lemma}[see {\cite[p. 53]{engel book 2000}}]\label{identify a core}
Let $\mathcal{A}$ be the infinitesimal generator of a $C_0$ semigroup $\mathcal{T}(t)$ on $X$. If $Y\subset D(\mathcal{A})$ is a dense subspace of $X$ and invariant under $\mathcal{T}(t)$ (i.e., $\mathcal{T}(t)Y\subset Y$), then $Y$ is a core for $\mathcal{A}$.
\end{lemma}

Next, we recall shortly how to use semigroups to solve abstract Cauchy problems. Suppose that $\mathcal{A}$ is the infinitesimal generator of a $C_0$ semigroup $e^{t\mathcal{A}}$ on a Banach space $(X,\|\cdot\|)$. We are concerned with the inhomogeneous abstract Cauchy problem
\begin{eqnarray}\label{ACP}
\left\{\begin{aligned}
&u'(t)-\mathcal{A}u(t)=f(t),\ \ 0<t\le T,\\
&u(0)=x.
\end{aligned}\right.
\end{eqnarray}
We assume that $x\in X$ and the inhomogeneous term $f$ only belongs to $L^1((0,T);X)$. Then \eqref{ACP} always has a unique {\it mild solution} $u\in C([0,T];X)$ given by the formula
\begin{equation*}
u(t)=e^{t\mathcal{A}}x+\int_0^t e^{(t-\tau)\mathcal{A}}f(\tau)\,d\tau.
\end{equation*}
A continuous function $u$ is called a {\it strong solution} if $u\in W^{1,1}((0,T);X)\cap L^1((0,T);D(\mathcal{A}))$ satisfies \eqref{ACP} for a.e. $t\in(0,T)$. A strong solution is also a mild solution. Conversely, a mild solution with suitable regularity becomes a strong one.

\begin{lemma}[see {\cite[Theorem~2.9]{pazy book 1983}}]\label{mild to strong}
Let $u\in C([0,T];X)$ be a mild solution to \eqref{ACP}. If $u\in W^{1,1}((0,T),X)$, or $u\in L^1((0,T),D(\mathcal{A}))$, then $u$ is a strong solution.
\end{lemma}

\subsection{Homogeneous Besov spaces}
In most literature on the theory of function spaces, the homogeneous Besov spaces are defined in the ambient space of tempered distributions modulo polynomials (see, e.g., \cite{Triebel book 1983}). However, we wish to avoid this type of spaces when solving nonlinear PDEs. In this paper, we adopt the definitions of homogeneous spaces in {\cite[Section~2.3]{BCD}}. Let $\mathscr{S}_{h}^{'}(\R^n)$ denote the space of all tempered distributions $u\in\mathscr{S}^{'}(\R^n)$ that satisfy
\begin{align*}
u=\sum_{ j\in \mathbb{Z}}\dot{\Delta }_{j}u\ \  \mathrm{in}\  \mathscr{S}^{'},
\end{align*}
where $\dot{\Delta}_{j}$'s are the homogeneous dyadic blocks (see {\cite[Chapter~2]{BCD}}).

\begin{definition}
Let $s\in\R$ and $1\leq p,r\leq\infty$. The homogeneous Besov space $\dot{B}_{p,r}^{s}(\R^n)$
consists of all distributions $u$ in $\mathscr{S}_{h}^{'}(\R^n)$ such that
$$
\|u\|_{\dot{B}_{p,r}^{s}}\vcentcolon=\left\|\big(2^{js}\|\dot{\Delta}_{j}u\|_{L^{p}}\big)_{j\in\mathbb{Z}}\right\|_{l^{r}}
<\infty.
$$
\end{definition}

Let us point out a couple of useful facts about the homogeneous Besov spaces. First, as an immediate consequence of the definition, we see that
\begin{align*}
\lim_{N\rightarrow\infty}\sum_{|j|<N}\dot{\Delta }_{j}u=u \ \ \mathrm{in}\ \dot{B}_{p,1}^{s}
\end{align*}
for every $u\in\dot{B}_{p,1}^{s}$. This implies that $W^{\infty,p}$ is dense in $\dot{B}_{p,1}^{s}$. Second, the spaces $\dot{B}_{p,r}^{s}$ are not always complete. In fact, $\dot{B}_{p,r}^{s}$ is complete only if $s<\frac{n}{p}$, or $s=\frac{n}{p}$ and $r=1$ (see {\cite[Theorem~2.25]{BCD}}). Finally, the product of two distributions in certain Besov spaces can be defined via the so-called {\it paraproduct}. In Section \ref{compressible pressureless flow section}, we will extensively use the product laws
\begin{align*}
\|uv\|_{\dot{B}_{p,1}^{n/p}}\lesssim\|u\|_{\dot{B}_{p,1}^{n/p}}\|v\|_{\dot{B}_{p,1}^{n/p}},\ \ 1\le p<\infty
\end{align*}
and
\begin{align*}
\|uv\|_{\dot{B}_{p,1}^{n/p-1}}\lesssim\|u\|_{\dot{B}_{p,1}^{n/p}}\|v\|_{\dot{B}_{p,1}^{n/p-1}},\ \ 1\le p<2n.
\end{align*}

It is sometimes quite useful to associate norms with operators arising from PDEs. For example, we can characterize the homogeneous Besov norms via the heat semigroups.
\begin{lemma}[see {\cite[Theorem~2.34]{BCD}}]\label{characterization via classic heat kernel}
Suppose that $s\in\R$ and $(p,q)\in[1,\infty]^2$. If $k>s/2$ and $k\ge0$, we have
\begin{align*}
\|u\|_{\dot{B}_{p,q}^{s,\Delta}}\vcentcolon=\left\|t^{-s/2}\|(t\Delta)^ke^{t\Delta}u\|_p\right\|_{L^q(\R_+;\frac{dt}{t})}\simeq\|u\|_{\dot{B}_{p,q}^{s}},\ \ \forall u\in\mathscr{S}_{h}^{'}.
\end{align*}
\end{lemma}

A similar result holds if we replace $\Delta$ by the Lam\'{e} operator. Here the Lam\'{e} operator $\mathcal{L}$ is defined by
\begin{align}\label{Lame operator}
\mathcal{L}\vcentcolon=\mu\Delta+(\lambda+\mu)\nabla\divg
\end{align}
with
\begin{align}\label{Lame coefficients}
\mu>0,\ \ \mathrm{and}\ \ \nu\vcentcolon=\lambda+2\mu>0.
\end{align}
Let us introduce the Hodge operator $\mathcal{Q}=-\nabla(-\Delta)^{-1}\divg$ and let $\mathcal{P}=I-\mathcal{Q}$. The Lam\'{e} operator $\mathcal{L}$ and the Laplacian $\Delta$ can be expressed by each other, namely,
\begin{align}\label{from Laplacian to Lame}
\mathcal{L}=(\mu\mathcal{P}+\nu\mathcal{Q})\Delta=\Delta(\mu\mathcal{P}+\nu\mathcal{Q})
\end{align}
and
\begin{align}\label{from Lame to Laplacian}
\Delta=\left(\frac{1}{\mu}\mathcal{P}+\frac{1}{\nu}\mathcal{Q}\right)\mathcal{L}=\mathcal{L}\left(\frac{1}{\mu}\mathcal{P}+\frac{1}{\nu}\mathcal{Q}\right).
\end{align}
So, for every $p\in(1,\infty)$ and $k\in\mathbb{N}$, we have
\begin{align}\label{equivalence between Lame and Laplacian}
\|\mathcal{L}^ku\|_p\simeq\|\Delta^k u\|_p,\ \ u\in W^{2k,p}(\R^n;\R^n).
\end{align}

\begin{lemma}\label{characterization via heat kernel of Lame operator}
Suppose that $s\in\R$, $p\in(1,\infty)$ and $q\in[1,\infty]$. If $k>s/2$ and $k\ge0$, we have
\begin{align*}
\|u\|_{\dot{B}_{p,q}^{s,\mathcal{L}}}\vcentcolon=\left\|t^{-s/2}\|(t\mathcal{L})^ke^{t\mathcal{L}}u\|_p\right\|_{L^q(\R_+;\frac{dt}{t})}\simeq\|u\|_{\dot{B}_{p,q}^{s}},\ \ \forall u\in L^p(\R^n;\R^n).
\end{align*}
\end{lemma}
\begin{proof}
The lemma is a consequence of Lemma \ref{characterization via classic heat kernel} along with the identities
\begin{align*}
e^{-t\mathcal{L}}=e^{\mu t\Delta}\mathcal{P}+e^{\nu t\Delta}\mathcal{Q}
\end{align*}
and
\begin{align*}
e^{t\Delta}=\mathcal{P}e^{-\mu^{-1}t\mathcal{L}}+\mathcal{Q}e^{-\nu^{-1}t\mathcal{L}}.
\end{align*}
\end{proof}

\bigskip

\section{Bounds of fundamental matrix}\label{heat kernel estimate section}

\bigskip

Let $\rho$ be a measurable function defined in $\R^n$ such that
\begin{align}\label{density bounds for Lame system}
m\le\rho(x)\le\frac{1}{m},\ \ \mathrm{a.e.}\ x\in\R^n
\end{align}
for some $m\in(0,1]$. Denote $b=\rho^{-1}$. The main results of this section, in the spirit of those in \cite{mcintosh 2000,duong DIE 1999}, are the Gaussian bounds of the matrix-valued heat kernel of $b\mathcal{L}$.

For notational convenience, we denote $L^2=L^2(\R^n;\R^n)$, $H^2=H^2(\R^n;\R^n)$. Let $\|\cdot\|$ be the $L^2$ norm induced by the standard $L^2$ inner product $\langle\cdot,\cdot\rangle$, and $\|\cdot\|_{\rho}$ the weighted norm induced by the inner product
\begin{align*}
\langle u,v\rangle_{\rho}=\int_{\R^n}u(x)\cdot v(x)\rho(x)\,dx.
\end{align*}

Roughly, our method is a classical PDE method, and we will study various weighted estimates for the solutions $u$ to the parabolic Lam\'{e} system
\begin{align}\label{Lame system}
\rho(x)\partial_tu-\mathcal{L}u=0,\ \ \mathrm{in}\ (0,\infty)\times\R^n.
\end{align}
Before studying the variable coefficient problem, let us point out a basic fact about the Lam\'{e} operator $\mathcal{L}$. The assumption \eqref{Lame coefficients} guarantees the ellipticity of $-\mathcal{L}$, and we have
\begin{align}\label{coercivity of Lame operator}
\|(-\mathcal{L})^{1/2}u\|^2=\langle-\mathcal{L}u,u\rangle=\mu\|\nabla u\|^2+(\mu+\lambda)\|\divg u\|^2\ge(\mu\wedge\nu)\|\nabla u\|^2
\end{align}
for all vectors $u\in H^2(\R^n;\R^n)$. 

\begin{lemma}\label{bL generates bounded analytic semigroup}
The operator $b\mathcal{L}:H^2\subset L^2\rightarrow L^2$ generates an analytic semigroup of contraction $\{e^{tb\mathcal{L}}\}_{t\ge0}$ on $(L^2,\langle\cdot,\cdot\rangle_\rho)$, and $e^{tb\mathcal{L}}b$ is self-adjoint on $(L^2,\langle\cdot,\cdot\rangle)$ for every $t\ge0$.
\end{lemma}
\begin{proof}
First, it is readily to verify that $b\mathcal{L}$ is a self-adjoint operator on $(L^2,\langle\cdot,\cdot\rangle_{\rho})$. In view of \eqref{coercivity of Lame operator}, we have $\langle b\mathcal{L}u,u\rangle_{\rho}\le0$. So by Theorem \ref{generation theorem of selfadjoint operator}, $b\mathcal{L}$ generates an analytic semigroup of contraction $\{e^{tb\mathcal{L}}\}_{t\ge0}$ on $(L^2,\langle\cdot,\cdot\rangle_{\rho})$. Since $e^{tb\mathcal{L}}$ is self-adjoint on $(L^2,\langle\cdot,\cdot\rangle_{\rho})$, we have for all $u,v\in L^2$ that
\begin{align*}
\langle e^{tb\mathcal{L}}bu,v\rangle=\langle e^{tb\mathcal{L}}bu,bv\rangle_{\rho}=\langle bu,e^{tb\mathcal{L}}bv\rangle_{\rho}=\langle u,e^{tb\mathcal{L}}bv\rangle.
\end{align*}
This means that $e^{tb\mathcal{L}}b$ is self-adjoint on $(L^2,\langle\cdot,\cdot\rangle)$.
\end{proof}

\begin{lemma}
Let $n\in\{2,3\}$. For every $t>0$, the bounded operator $e^{tb\mathcal{L}}$ on $L^2$ admits a Schwartz kernel, denoted by $K_t(x,y)$, which is bounded and satisfies the pointwise bound
\begin{align*}
|K_t(x,y)|\le\frac{C}{t^{n/2}}
\end{align*}
for some constant $C=C(m,\mu,\lambda)$.
\end{lemma}
\begin{proof}
Since $n\in\{2,3\}$, we get from the Gagliardo-Nirenberg inequality and \eqref{equivalence between Lame and Laplacian} that
\begin{align}\label{GN inequality interpolate L infty}
\|u\|_{\infty}\le C\|u\|^{1-n/4}\|\mathcal{L}u\|^{n/4}\ \ u\in H^2.
\end{align}
This along with the analyticity of $e^{tb\mathcal{L}}$ implies that
\begin{align*}
\|e^{tb\mathcal{L}}u_0\|_{\infty}\le Ct^{-n/4}\|u_0\|,\ \ u_0\in L^2.
\end{align*}
Then $e^{tb\mathcal{L}}$ is also bounded from $L^1$ to $L^2$ due to the self-adjointness of $e^{tb\mathcal{L}}b$, and from $L^1$ to $L^\infty$ due to the semigroup property. So the Schwartz kernel $K_t(x,y)$ of $e^{tb\mathcal{L}}$ is indeed bounded and satisfies the desired bound. This completes the proof.
\end{proof}

Next, we adopt the well-known Davies perturbation method (see \cite{davies ajm 1987}) to show Gaussian bounds for the kernel $S_t(x,y)\vcentcolon=K_t(x,y)b(y)$. 

The main theorem in this section is the following:
\begin{theorem}\label{main theorem of heat kernel estimates}
Let $n\in\{2,3\}$. For any $\gamma\in(0,1)$, each entry of $S_t(x,y)$ is a $C^{1,\gamma}$ function in both $x$ and $y$. More precisely, there exist constants $C_1=C_1(m,\mu,\lambda)$ and $C_2=C_2(m,\mu,\lambda,\gamma)$ such that for all $t>0$ and $x,y,h\in\R^n$,
\begin{align}
|S_t(x,y)|+\sqrt{t}|\nabla_x S_t(x,y)|\le\frac{C_1}{t^{n/2}}\exp\left\{-\frac{|x-y|^2}{C_1t}\right\},
\end{align}
\begin{align}\label{Holder estimate for gradient 1}
|\nabla_x S_t(x+h,y)-\nabla_xS_t(x,y)|\le\left(\frac{|h|}{\sqrt{t}}\right)^{\gamma}\frac{C_2}{t^{(n+1)/2}}\exp\left\{-\frac{|x-y|^2}{C_2t}\right\},
\end{align}
and
\begin{align}\label{Holder estimate for gradient 2}
|\nabla_x S_t(x,y+h)-\nabla_xS_t(x,y)|\le\left(\frac{|h|}{\sqrt{t}}\right)^{\gamma}\frac{C_2}{t^{(n+1)/2}}\exp\left\{-\frac{|x-y|^2}{C_2t}\right\}
\end{align}
provided $2|h|\le\sqrt{t}$.
\end{theorem}
\begin{remark}
In view of Lemma \ref{bL generates bounded analytic semigroup}, we have $S_t(x,y)=S_t^{\intercal}(y,x)$. So the $y$-derivative also satisfies each of the bounds.
\end{remark}

Let $\mathscr{W}$ denote the set of all bounded real-valued smooth functions $\psi$ on $\R^n$ such that $\|\nabla\psi\|_{\infty}\le1$ and $\|\nabla^2\psi\|_{\infty}\le1$. Let $d(x,y)\vcentcolon=\sup\{\psi(x)-\psi(y)|\psi\in\mathscr{W}\}$.

\begin{lemma}[see {\cite[Lemma~4]{davies jfa 1995}}]\label{equivalence of metric}
There exists a positive constant $C=C(n)$ such that
\begin{align*}
C^{-1}|x-y|\le d(x,y)\le C|x-y|
\end{align*}
for all $x,y\in\R^n$.
\end{lemma}

Given $\alpha\in\R$ and $\psi\in\mathscr{W}$, define $\psi_\alpha(x)=\psi(\alpha x)$ and  $\phi(x)=e^{\psi_\alpha(x)}$. The analysis is based on the key observation that
\begin{align}\label{variation of purterbation Lame operator}
\langle-\phi^{-1}\mathcal{L}\phi v,u\rangle=&\mu\int(\alpha(\nabla\psi)_\alpha\otimes v+\nabla v):(-\alpha(\nabla\psi)_\alpha\otimes u+\nabla u)\,dx\nonumber\\
&+(\mu+\lambda)\int(\alpha(\nabla\psi)_\alpha\cdot v+\divg v)(-\alpha(\nabla\psi)_\alpha\cdot u+\divg u)\,dx
\end{align}
for any smooth vector fields $u$ and $v$. In particular, if $u=v$, we have
\begin{align}\label{coercivity of purterbation Lame operator}
\langle-\phi^{-1}\mathcal{L}\phi v,v\rangle\ge\|(-\mathcal{L})^{1/2}v\|^2-C\alpha^2\|v\|^2.
\end{align}

In what follows, we divide the proof of Theorem \ref{main theorem of heat kernel estimates} into three lemmas.
\begin{lemma}\label{first step Gaussian upper bound}
Let $n\in\{2,3\}$. There exists a constant $C=C(m,\mu,\lambda)$ such that for all $t>0$ and $x,y\in\R^n$,
\begin{align*}
|K_t(x,y)|\le \frac{C}{t^{n/2}}\exp\left\{-\frac{|x-y|^2}{Ct}\right\}.
\end{align*}
\end{lemma}
\begin{proof}
Denote $v=\phi^{-1}e^{tb\mathcal{L}}(\phi u_0)$, where $u_0\in L^2$. Then $v$ is a solution to the system
\begin{eqnarray}\label{equation for v}
\left\{\begin{aligned}
&\rho\partial_t v-\phi^{-1}\mathcal{L}\phi v=0,\\
&v(0)=u_0.
\end{aligned}\right.
\end{eqnarray}

We start with the energy estimates for $v$. Taking inner product of \eqref{equation for v} with $v$, then using \eqref{coercivity of purterbation Lame operator}, we get
\begin{align*}
\frac{1}{2}\frac{d}{dt}\|v\|_\rho^2+\|(-\mathcal{L})^{1/2}v\|^2\le C\alpha^2\|v\|_\rho^2.
\end{align*}
Applying Gronwall's inequality, we obtain
\begin{align}\label{L2 energy estimate for purterbation problem}
\|v(t)\|^2+\int_0^t\|(-\mathcal{L})^{1/2}v\|^2\,d\tau\le C\|u_0\|^2e^{C\alpha^2t}.
\end{align}
Differentiating \eqref{equation for v} with respect to $t$, we get by a similar argument that
\begin{align*}
\frac{1}{2}\frac{d}{dt}\|\partial_tv\|_\rho^2+\|(-\mathcal{L})^{1/2}\partial_tv\|^2\le C\alpha^2\|\partial_tv\|_\rho^2.
\end{align*}
So the function $t\mapsto\|\partial_tv\|_\rho^2e^{-C\alpha^2t}$ is decreasing. Consequently, we have
\begin{align}\label{time derivative energy estimate 1 for purterbation problem}
\|\partial_tv\|_\rho^2e^{-C\alpha^2t}\le \frac{2}{t}\int_{t/2}^{t}\|\partial_tv\|_\rho^2e^{-C\alpha^2\tau}\,d\tau.
\end{align}
Next, multiplying \eqref{equation for v} by $v_t$ and integrating in $x$, then using \eqref{variation of purterbation Lame operator} and the Cauchy-Schwarz inequality, we get
\begin{align*}
\|\partial_tv\|_\rho^2+\frac12\frac{d}{dt}\|(-\mathcal{L})^{1/2}v\|^2\le C\left(\alpha^2\|v\|\|\partial_tv\|+|\alpha|\|\nabla v\|\|\partial_tv\|\right).
\end{align*}
The term $\|\partial_tv\|$ on the right side can be absorbed by $\|\partial_tv\|_\rho^2$ on the left side. This together with \eqref{L2 energy estimate for purterbation problem} and \eqref{coercivity of Lame operator} gives
\begin{align*}
\|\partial_tv\|_\rho^2+\frac{d}{dt}\|(-\mathcal{L})^{1/2}v\|^2\le C\left(\alpha^4e^{C\alpha^2t}\|u_0\|^2+\alpha^2\|(-\mathcal{L})^{1/2}v\|^2\right).
\end{align*}
So,
\begin{align}\label{time derivative energy estimate 2 for purterbation problem}
\|\partial_tv\|_\rho^2e^{-C\alpha^2t}+\frac{d}{dt}(\|(-\mathcal{L})^{1/2}v\|^2e^{-C\alpha^2t})\le C\alpha^4\|u_0\|^2.
\end{align}
Combing \eqref{L2 energy estimate for purterbation problem} and \eqref{time derivative energy estimate 2 for purterbation problem}, we have
\begin{align*}
\int_{t/2}^{t}\|\partial_tv\|_\rho^2e^{-C\alpha^2\tau}\,d\tau+\|(-\mathcal{L})^{1/2}v(t)\|^2\le C\left(\frac1t+\alpha^4t\right)\|u_0\|^2e^{C\alpha^2t},
\end{align*}
which together with \eqref{time derivative energy estimate 1 for purterbation problem} further implies
\begin{align}\label{time derivative energy estimate 3 for purterbation problem}
\|\partial_tv(t)\|\le C\left(\alpha^2+\frac{1}{t}\right)\|u_0\|e^{C\alpha^2t}\le \frac{C}{t}\|u_0\|e^{C\alpha^2t}.
\end{align}

The above estimate should imply the corresponding $L^2$ estimate of $\mathcal{L}v$. To see this, we get by a direct computation that
\begin{align}\label{expression for Lv}
-\mathcal{L}v=&-\rho\partial_tv+\mu(\alpha^2|\nabla\psi|_\alpha^2v+2\alpha(\nabla\psi)_\alpha\cdot\nabla v+\alpha^2(\Delta\psi)_\alpha v)\nonumber\\
&+(\mu+\lambda)(\alpha\divg v(\nabla\psi)_\alpha+\alpha\nabla v(\nabla\psi)_\alpha+\alpha^2(v\cdot(\nabla\psi)_\alpha)(\nabla\psi)_\alpha+\alpha^2(\nabla^2\psi)_{\alpha}v).
\end{align}
Then it is easy to see that
\begin{align*}
\|\mathcal{L}v\|\le C(\|\partial_tv\|+\alpha^2\|v\|+|\alpha|\|\nabla v\|).
\end{align*}
The first order derivative can be handled by using the interpolation inequality
\begin{align*}
\|\nabla v\|\le C\|v\|^{1/2}\|\mathcal{L} v\|^{1/2}.
\end{align*}
So,
\begin{align*}
\|\mathcal{L}v\|\le C(\|\partial_tv\|+\alpha^2\|v\|).
\end{align*}
Substituting for $\|v\|$ and $\|\partial_tv\|$ by \eqref{L2 energy estimate for purterbation problem} and \eqref{time derivative energy estimate 3 for purterbation problem}, respectively, we have
\begin{align*}
\|\mathcal{L}v(t)\|\le \frac{C}{t}\|u_0\|e^{C\alpha^2t}.
\end{align*}

Now using the Gagliardo-Nirenberg inequality \eqref{GN inequality interpolate L infty}, we obtain
\begin{align}\label{from L2 to Linfty}
\|v(t)\|_\infty\le \frac{C}{t^{n/4}}\|u_0\|e^{C\alpha^2t}.
\end{align}
This means the operator $\phi^{-1}e^{tb\mathcal{L}}\phi$ is bounded from $L^2$ to $L^\infty$. A duality argument gives the bound from $L^1$ to $L^2$
\begin{align}\label{from L1 to L2}
\|v(t)\|\le \frac{C}{t^{n/4}}\|u_0\|_1e^{C\alpha^2t}.
\end{align}
While this along with the semigroup property of $\phi^{-1}e^{tb\mathcal{L}}\phi$ gives
\begin{align}\label{from L1 to Linfty}
\|v(t)\|_\infty\le \frac{C}{t^{n/2}}\|u_0\|_1e^{C\alpha^2t}.
\end{align}
Noticing that the kernel of $\phi^{-1}e^{tb\mathcal{L}}\phi$ is $K_t(x,y)e^{\psi(\alpha y)-\psi(\alpha x)}$, we get
\begin{align*}
|K_t(x,y)|\le\frac{C}{t^{n/2}}\exp\{C\alpha^2t+\psi(\alpha x)-\psi(\alpha y)\}.
\end{align*}
Replacing $\psi$ by $-\psi$, we have
\begin{align*}
|K_t(x,y)|\le\frac{C}{t^{n/2}}\exp\{C\alpha^2t-|\psi(\alpha x)-\psi(\alpha y)|\}.
\end{align*}
It follows by optimizing with respect to $\psi\in\mathscr{W}$ and applying Lemma \ref{equivalence of metric} that
\begin{align*}
|K_t(x,y)|\le\frac{C}{t^{n/2}}\exp\{C\alpha^2t-C^{-1}|\alpha||x-y|\}.
\end{align*}
Finally, minimizing the bound by choosing $\alpha=\frac{|x-y|}{2C^2t}$ completes the proof.
\end{proof}

\begin{lemma}\label{second step gradient estimate}
Let $n\in\{2,3\}$. There exists a constant $C=C(m,\mu,\lambda)$ such that for all $t>0$ and $x,y\in\R^n$,
\begin{align*}
|\nabla_x K_t(x,y)|\le \frac{C}{t^{(n+1)/2}}\exp\left\{-\frac{|x-y|^2}{Ct}\right\}.
\end{align*}
\end{lemma}
\begin{proof}
Apparently, we only need to show the bound for $|\nabla_x K_t(x,y)|$. Denote $u=e^{tb\mathcal{L}}(\phi u_0)$ and $v=\phi^{-1}u$, where $u_0\in L^2$. We need to bound the norm $\|\phi^{-1}\nabla u(t)\|_\infty$. To this end, let us first study the norm $\|\nabla v(t)\|_\infty$ since 
\begin{align}\label{express u in terms of v}
\phi^{-1}\nabla u(t)=\nabla v+\alpha(\nabla\psi)_{\alpha}\otimes v.
\end{align}

By the equation \eqref{expression for Lv}, we see that
\begin{align}\label{L infty estimate for Lv}
\|\mathcal{L}v\|_\infty\le C(\|\partial_tv\|_\infty+\alpha^2\|v\|_\infty+|\alpha|\|\nabla v\|_\infty).
\end{align}
Using Littlewood-Paley and \eqref{from Lame to Laplacian}, one can prove the interpolation inequality
\begin{align}\label{GN inequality via Littlewood Paley}
\|\nabla v\|_\infty\le C\| v\|_\infty^{1/2}\|\mathcal{L}v\|_\infty^{1/2}.
\end{align}
Plugging \eqref{GN inequality via Littlewood Paley} in \eqref{L infty estimate for Lv} gives
\begin{align}\label{L infty estimate for Lv 2}
\|\mathcal{L}v\|_\infty\le C(\|\partial_tv\|_\infty+\alpha^2\|v\|_\infty).
\end{align}
Then combining \eqref{express u in terms of v}, \eqref{GN inequality via Littlewood Paley} and \eqref{L infty estimate for Lv 2}, we arrive at
\begin{align}\label{Linfty gradient estimate for u}
\|\phi^{-1}\nabla u(t)\|_\infty\le C\left(|\alpha|\|v(t)\|_\infty+\|v(t)\|_\infty^{1/2}\|\partial_tv(t)\|_{\infty}^{1/2}\right).
\end{align}

Next, in order to bound $\|\partial_tv(t)\|_{\infty}$, we observe that
\begin{align*}
\partial_tv(t)=\phi^{-1}e^{\frac{t}{2}b\mathcal{L}}\phi[\partial_tv(t/2)].
\end{align*}
So, in view of \eqref{from L2 to Linfty}, \eqref{time derivative energy estimate 3 for purterbation problem} and \eqref{from L1 to L2}, we get
\begin{align}\label{L1 to Linfty estimate for time derivative}
\|\partial_tv(t)\|_{\infty}\le\frac{C}{t^{n/4}}e^{C\alpha^2t}\|\partial_tv(t/2)\|\le\frac{C}{t^{1+n/4}}e^{C\alpha^2t}\|v(t/4)\|\le\frac{C}{t^{1+n/2}}e^{C\alpha^2t}\|u_0\|_1.
\end{align}
Plugging the above in \eqref{Linfty gradient estimate for u} and using \eqref{from L1 to Linfty}, we have
\begin{align}
\|\phi^{-1}\nabla u(t)\|_\infty\le C\left(\frac{|\alpha|}{t^{n/2}}+\frac{1}{t^{(n+1)/2}}\right)e^{C\alpha^2t}\|u_0\|_1\le\frac{C}{t^{(n+1)/2}}e^{C\alpha^2t}\|u_0\|_1.
\end{align}
Thus,
\begin{align*}
|\nabla_x K_t(x,y)|\le \frac{C}{t^{(n+1)/2}}\exp\{C\alpha^2t+\psi(\alpha x)-\psi(\alpha y)\}.
\end{align*}
Again, we finish the proof by optimizing the bound with respect to $\psi\in\mathscr{W}$ and then $\alpha\in\R$.
\end{proof}
\begin{remark}\label{remark on analyticity on Lp}
From \eqref{L1 to Linfty estimate for time derivative}, we also see that the kernel of $tb\mathcal{L}e^{tb\mathcal{L}}$ has a pointwise Gaussian upper bound. In particular, $tb\mathcal{L}e^{tb\mathcal{L}}$ extends to a bounded operator on $L^p$ for every $t>0$.
\end{remark}

\begin{lemma}\label{last step Holder estimate for gradient}
Let $n\in\{2,3\}$. For any $\gamma\in(0,1)$, there exists a constant $C=C(m,\mu,\lambda,\gamma)$ such that for all $t>0$ and $x,y,h\in\R^n$, \eqref{Holder estimate for gradient 1} and \eqref{Holder estimate for gradient 2} hold whenever $2|h|\le\sqrt{t}$.
\end{lemma}
\begin{proof}
Let $u=e^{tb\mathcal{L}}u_0$. By Lemmas \ref{bL generates bounded analytic semigroup} and \ref{first step Gaussian upper bound}, we have
\begin{align*}
t^{n/4}\|\mathcal{L}u(t)\|+t^{n/2}\|\mathcal{L}u(t)\|_\infty\le\frac{C}{t}\|u_0\|_1.
\end{align*}
For any $\gamma\in(0,1)$, let $q=\frac{n}{1-\gamma}$ and $\theta=\frac{2(1-\gamma)}{n}$. Then we use the embedding $\Dot{W}^{1,q}(\R^n)\hookrightarrow \dot{C}^{\gamma}(\R^n)$ to get
\begin{align*}
\|\nabla u\|_{\dot{C}^{\gamma}}\le C\|\nabla^2u\|_q\le C\|\mathcal{L}u\|_q\le C\|\mathcal{L}u\|_{2}^{\theta}\|\mathcal{L}u\|_{\infty}^{1-\theta}\le\frac{C}{t^{(n+1+\gamma)/2}}\|u_0\|_1.
\end{align*}
Thus, we have for any $h\in\R^n$ that
\begin{align*}
|\nabla_x K_t(x+h,y)-\nabla_x K_t(x,y)|\le\frac{C}{t^{(n+1)/2}}\left(\frac{|h|}{\sqrt{t}}\right)^{\gamma}.
\end{align*}
The exponential decay factor in \eqref{Holder estimate for gradient 1} can be easily obtained by the observation that
\begin{align*}
&|\nabla_x K_t(x+h,y)-\nabla_x K_t(x,y)|\\
\le&(|\nabla_x K_t(x+h,y)|+|\nabla_x K_t(x,y)|)^{1-\beta}|\nabla_x K_t(x+h,y)-\nabla_x K_t(x,y)|^{\beta}
\end{align*}
for any $\beta\in(0,1)$. This proves \eqref{Holder estimate for gradient 1}.

To prove $\eqref{Holder estimate for gradient 2}$, we write
\begin{align*}
\int[\nabla_x S_t(x,y+h)-\nabla_x S_t(x,y)]u_0(y)\,dy=\nabla e^{tb\mathcal{L}}(b\delta_hu_0)
\end{align*}
with $\delta_hu_0(x)=u_0(x-h)-u_0(x)$. Using Lemma \ref{second step gradient estimate}, the right side can be estimated as follows
\begin{align*}
\|\nabla e^{tb\mathcal{L}}(b\delta_hu_0)\|_{\infty}\le\frac{C}{\sqrt{t}}\|e^{\frac{t}{2}b\mathcal{L}}(b\delta_hu_0)\|_{\infty}\le\frac{C|h|}{t^{1+n/2}}\|u_0\|_1.
\end{align*}
The bound in \eqref{Holder estimate for gradient 2} can be shown by a similar argument as the first part of the proof. This completes the proof of the lemma.
\end{proof}

For completeness, we conclude this section by finishing the proof of Theorem \ref{main theorem of heat kernel estimates}.
\begin{proof}[Proof of Theorem \ref{main theorem of heat kernel estimates}]
Lemmas \ref{first step Gaussian upper bound}-\ref{last step Holder estimate for gradient} constitute the proof of Theorem \ref{main theorem of heat kernel estimates}.
\end{proof}

\bigskip

\section{An abstract $L^1$ theory}\label{maximal regularity for ACP section}

\bigskip

In this section, we are concerned with the $L^1$-in-time theory for the abstract Cauchy problem \eqref{ACP associated with composite operator}, where $\mathcal{S}$ is a composition of bounded and unbounded operator. We follow our prior work \cite{huan 2021} closely and we do not explicitly use the theory of interpolation spaces. 

Let $(X,\|\cdot\|)$ be a Banach space. We temporarily just assume
\begin{assumption}\label{assume S is injective and generates bounded analytic semigroup}
$\mathcal{S}:D(\mathcal{S})\subset X\rightarrow X$ is an one-to-one operator that generates a bounded analytic semigroup $e^{t\mathcal{S}}$ on $X$.  
\end{assumption}

Given $s\in(0,2)$, we define
\begin{align*}
\|x\|_{\dot{B}_{X,1}^{s,\mathcal{S}}}\vcentcolon=\|t^{-s/2}\|t\mathcal{S}e^{t\mathcal{S}}x\|\|_{L^1(\R_+,\frac{dt}{t})}
\end{align*}
and
\begin{align*}
\|x\|_{\dot{B}_{X,1}^{-s,\mathcal{S}}}\vcentcolon=\|t^{s/2}\|e^{t\mathcal{S}}x\|\|_{L^1(\R_+,\frac{dt}{t})}.
\end{align*}
In view of Lemmas \ref{characterization via classic heat kernel} and \ref{characterization via heat kernel of Lame operator}, the above notations make sense if we pretend that $\mathcal{S}$ is a second-order elliptic operator. For any $x\in D(\mathcal{S})$, since $\|t\mathcal{S}e^{t\mathcal{S}}x\|\lesssim\|x\|\wedge\|t\mathcal{S}x\|$, we easily see that
\begin{align*}
\|x\|_{\dot{B}_{X,1}^{s,\mathcal{S}}}\lesssim\|x\|_{D(\mathcal{S})}\vcentcolon=\|x\|+\|\mathcal{S}x\|.
\end{align*}
While for $x\in R(\mathcal{S})$, the range of $\mathcal{S}$, we have
\begin{align*}
\|x\|_{\dot{B}_{X,1}^{-s,\mathcal{S}}}=\|\mathcal{S}^{-1}x\|_{\dot{B}_{X,1}^{2-s,\mathcal{S}}}\lesssim\|x\|_{R(\mathcal{S})}\vcentcolon=\|x\|+\|\mathcal{S}^{-1}x\|.
\end{align*}
\begin{definition}
Define $\dot{B}_{X,1}^{s,\mathcal{S}}$ as the completion of $(D(\mathcal{S}),\|\cdot\|_{\dot{B}_{X,1}^{s,\mathcal{S}}})$, and $\dot{B}_{X,1}^{-s,\mathcal{S}}$ as the completion of $(R(\mathcal{S}),\|\cdot\|_{\dot{B}_{X,1}^{-s,\mathcal{S}}})$.
\end{definition}
The space $\dot{B}_{X,1}^{s,\mathcal{S}}$ can also be defined via interpolation (see {\cite[Remark~2.4]{giga jfa 1991}}), but we do not need this fact in this paper.

For notational convenience, we temporarily denote $\dot{B}_{X,1}^{\pm s,\mathcal{S}}$ by $\dot{B}^{\pm s}$. But we shall not use the abbreviated notations if the norms are associated with different operators.
\begin{lemma}
For every $k\in\mathbb{N}\cup\{0\}$ and $s\in(0,2)$, there exists a constant $C$ depending on $s$ and $k$ such that
\begin{align}\label{uniform bounded in B positive s}
\sup_{t>0}\|(t\mathcal{S})^ke^{t\mathcal{S}}x\|_{\dot{B}^{s}}\le C\|x\|_{\dot{B}^{s}},\ \ \forall x\in D(\mathcal{S}),
\end{align}
\begin{align}\label{uniform bounded in B negative s}
\sup_{t>0}\|(t\mathcal{S})^ke^{t\mathcal{S}}x\|_{\dot{B}^{-s}}\le C\|x\|_{\dot{B}^{-s}},\ \ \forall x\in R(\mathcal{S}),
\end{align}
\begin{align}\label{global in time bounded in B positive s}
\left\|\|(t\mathcal{S})^{k+1} e^{t\mathcal{S}}x\|_{\dot{B}^{s}}\right\|_{L^1(\R_+,\frac{dt}{t})}\le C\|x\|_{\dot{B}^{s}},\ \ x\in D(\mathcal{S}),
\end{align}
and
\begin{align}\label{global in time bounded in B negative s}
\left\|\|(t\mathcal{S})^{k+1} e^{t\mathcal{S}}x\|_{\dot{B}^{-s}}\right\|_{L^1(\R_+,\frac{dt}{t})}\le C\|x\|_{\dot{B}^{-s}},\ \ x\in R(\mathcal{S}).
\end{align}
\end{lemma}
\begin{proof}
The first two inequalities follow immediately from the definitions of the norms and the analyticity of $\mathcal{S}$.

The proofs for \eqref{global in time bounded in B positive s} and \eqref{global in time bounded in B negative s} are similar, so let us only prove \eqref{global in time bounded in B positive s}. In view of \eqref{uniform bounded in B positive s}, we only need to prove \eqref{global in time bounded in B positive s} for $k=0$. Applying Fubini's theorem, we have
\begin{align*}
\int_0^\infty\|\tau\mathcal{S}e^{\tau\mathcal{S}}x\|_{\dot{B}^{s}}\,\frac{d\tau}{\tau}=&\int_0^\infty\int_{0}^{\infty}t^{-s/2}\|\mathcal{S}^2e^{(t+\tau)\mathcal{S}}x\| \,dt\,d\tau\\
=&\int_0^\infty\int_{\tau}^{\infty}(t-\tau)^{-s/2}\|\mathcal{S}^2e^{t\mathcal{S}}x\|\,dt\,d\tau\\
=&\int_0^\infty\|\mathcal{S}^2e^{t\mathcal{S}}x\|\,dt\int_{0}^{t}(t-\tau)^{-s/2}\,d\tau\\
=&\frac{2}{2-s}\int_0^\infty t^{-s/2}\|t\mathcal{S}^2e^{t\mathcal{S}}x\|\,dt.
\end{align*}
Finally, by the analyticity of $\mathcal{S}$, we end up with
\begin{align*}
\le C\int_0^\infty t^{-s/2}\|\mathcal{S}e^{t\mathcal{S}}x\| \,dt
= C\|x\|_{\dot{B}^{s}}.
\end{align*}
This completes the proof.
\end{proof}

The inequality \eqref{uniform bounded in B positive s} (with $k=0$) guarantees that $e^{t\mathcal{S}}|_{D(\mathcal{S})}$ extends to a bounded operator on $\dot{B}^{s}$ with bounds uniform in $t$. Denote this extension by $\mathcal{T}_{s}(t)$. Then $\{\mathcal{T}_{s}(t)\}_{t\ge0}$ is a bounded semigroup on $\dot{B}^{s}$. Similarly, \eqref{uniform bounded in B negative s} implies that $e^{t\mathcal{S}}$ also extrapolates to a bounded semigroup $\{\mathcal{T}_{-s}(t)\}_{t\ge0}$ on $\dot{B}^{-s}$. In fact, both semigroups are strongly continuous.
\begin{lemma}
$\{\mathcal{T}_{s}(t)\}_{t\ge0}$ (resp., $\{\mathcal{T}_{-s}(t)\}_{t\ge0}$) is a bounded $C_0$ semigroup on $\dot{B}^{s}$ (resp., $\dot{B}^{-s}$).
\end{lemma}
\begin{proof}
For $x\in D(\mathcal{S})$, the function $t\mapsto\mathcal{T}_{s}(t)x=e^{t\mathcal{S}}x$ belongs to $C([0,\infty);D(\mathcal{S}))$, hence $C([0,\infty);\dot{B}^{s})$ since $D(\mathcal{S})\hookrightarrow\dot{B}^{s}$. One can easily get the strong continuity of $\mathcal{T}_{s}(t)$ on $\dot{B}^{s}$ by a density argument.

The strong continuity of $\mathcal{T}_{-s}(t)$ on $\dot{B}^{-s}$ can be proved analogously.
\end{proof}

Let us denote by $\mathcal{G}_{s}$ and $\mathcal{G}_{-s}$ the generators of $\mathcal{T}_{s}(t)$ and $\mathcal{T}_{-s}(t)$, respectively. In general, it is not easy to identify the domain of the generator of a semigroup. However, it would be easier to find a core for the generator.

\begin{lemma}\label{G is a closure of S}
{\rm (\romannumeral1)} The domain $D(\mathcal{S}^2)$ of $\mathcal{S}^2$ is a core for $\mathcal{G}_s$, and it holds that $\mathcal{G}|_{D(\mathcal{S}^2)}=\mathcal{S}|_{D(\mathcal{S}^2)}$, that is, $\mathcal{G}_{s}$ is the closure of $\mathcal{S}:D(\mathcal{S}^2)\subset\Dot{B}^s\rightarrow\Dot{B}^s$.

{\rm (\romannumeral2)} $\mathcal{G}_{-s}$ is the closure of $\mathcal{S}:D(\mathcal{S})\cap R(\mathcal{S})\subset\Dot{B}^{-s}\rightarrow\Dot{B}^{-s}$.
\end{lemma}
\begin{proof}
Note that $D(\mathcal{S}^2)$ is dense in $\dot{B}^{s}$ since $D(\mathcal{S}^2)$ is dense in $D(\mathcal{S})$ and $D(\mathcal{S})$ is dense in $\dot{B}^{s}$. For every $x\in D(\mathcal{S}^2)$, we have
\begin{align*}
\frac1t (\mathcal{T}_{s}(t)x-x)=\frac1t\int_0^t\mathcal{T}_{s}(\tau)\mathcal{S}x\,d\tau.
\end{align*}
Letting $t\rightarrow0^+$, the right side converges to $\mathcal{S}x$ in $D(\mathcal{S})$, thus, in $\Dot{B}^s$. We infer that $D(\mathcal{S}^2)\subset D(\mathcal{G}_{s})$ and $\mathcal{G}_{s}|_{D(\mathcal{S}^2)}=\mathcal{S}|_{D(\mathcal{S}^2)}$. Obviously, $D(\mathcal{S}^2)$ is invariant under $\mathcal{T}_{s}(t)$. Thus, by Lemma \ref{identify a core}, $D(\mathcal{S}^2)$ is a core for $\mathcal{G}_{s}$.

We prove the second part along the lines of the above proof. First, $D(\mathcal{S})\cap R(\mathcal{S})$ is dense in $\dot{B}^{-s}$ since $D(\mathcal{S})\cap R(\mathcal{S})$ is dense in $(R(\mathcal{S}),\|\cdot\|_{R(\mathcal{S})})$ and $R(\mathcal{S})$ is dense in $\dot{B}^{-s}$. Next, we can show that $D(\mathcal{S})\cap R(\mathcal{S})\subset D(\mathcal{G}_{-s})$ and $\mathcal{G}_{-s}|_{D(\mathcal{S})\cap R(\mathcal{S})}=\mathcal{S}|_{D(\mathcal{S})\cap R(\mathcal{S})}$. Moreover, since $D(\mathcal{S})\cap R(\mathcal{S})$ is invariant under $\mathcal{T}_{-s}(t)$, so it is a core for $\mathcal{G}_{-s}$. This completes the proof.
\end{proof}

\begin{lemma}\label{T is a bounded analytic semigroup}
$\{\mathcal{T}_{s}(t)\}_{t\ge0}$ (resp., $\{\mathcal{T}_{-s}(t)\}_{t\ge0}$) is a bounded analytic semigroup on $\dot{B}^{s}$ (resp., $\dot{B}^{-s}$).
\end{lemma}
\begin{proof}
We know from Lemma \ref{G is a closure of S} that $D(\mathcal{S}^2)$ is dense in $\dot{B}^{s}$, and that $\mathcal{G}_{s}\mathcal{T}_{s}(t)x=\mathcal{S}e^{t\mathcal{S}}x$ for $x\in D(\mathcal{S}^2)$. It then follows from \eqref{uniform bounded in B positive s} that $\|t\mathcal{G}_{s}\mathcal{T}_{s}(t)x\|_{\dot{B}^{s}}\le C\|x\|_{\dot{B}^{s}}$ for every $t>0$. So $\mathcal{T}_{s}(t)$ is a bounded analytic semigroup. An analogous argument gives the analyticity of $\mathcal{T}_{-s}(t)$ on $\dot{B}^{-s}$.
\end{proof}

\begin{remark}
By Fatou's lemma, now \eqref{global in time bounded in B positive s} (resp., \eqref{global in time bounded in B negative s}) actually holds for data in $\dot{B}^{s}$ (resp., $\dot{B}^{-s}$). In particular, choosing $k=0$, we have
\begin{align}\label{maximal L1 regularity criterion positive s}
\|\mathcal{G}_{s} e^{t\mathcal{G}_{s}}x\|_{L^1(\R_+,\dot{B}^{s})}\le C\|x\|_{\dot{B}^{s}},\ \ \forall x\in\dot{B}^{s}
\end{align}
and
\begin{align}\label{maximal L1 regularity criterion negative s}
\|\mathcal{G}_{-s} e^{t\mathcal{G}_{-s}}x\|_{L^1(\R_+,\dot{B}^{-s})}\le C\|x\|_{\dot{B}^{-s}},\ \ \forall x\in\dot{B}^{-s}.
\end{align}
\end{remark}

Next, we take advantage of Lemma \ref{T is a bounded analytic semigroup}, \eqref{maximal L1 regularity criterion positive s} and \eqref{maximal L1 regularity criterion negative s} to obtain the maximal $L^1$ regularity for the abstract Cauchy problems
\begin{align}\label{ACP positive s}
u'(t)-\mathcal{G}_{s}u(t)=f(t),\ \ u(0)=x
\end{align}
and
\begin{align}\label{ACP negative s}
u'(t)-\mathcal{G}_{-s}u(t)=f(t),\ \ u(0)=x.
\end{align}

\begin{theorem}\label{Maximal regularity theorem associated with Gs}
Assume Assumption \ref{assume S is injective and generates bounded analytic semigroup}. Let $s\in(0,2)$ and $T\in(0,\infty]$. There exists a constant $C=C(s)$ such that

{\rm (\romannumeral1)} For any $x\in\dot{B}^{s}$ and $f\in L^1((0,T);\dot{B}^{s})$, the equation \eqref{ACP positive s} has a unique strong solution $u\in C([0,T);\dot{B}^{s})$ satisfying
\begin{align*}
\|u\|_{L_T^\infty(\dot{B}^{s})}+\|u',\mathcal{G}_su\|_{L_T^1(\dot{B}^{s})}\le C\|x\|_{\dot{B}^{s}}+C\|f\|_{L_T^1(\dot{B}^{s})}.
\end{align*}

{\rm (\romannumeral2)} For any $x\in\dot{B}^{-s}$ and $f\in L^1((0,T);\dot{B}^{-s})$, the equation \eqref{ACP negative s} has a unique strong solution $u\in C([0,T);\dot{B}^{-s})$ satisfying
\begin{align*}
\|u\|_{L_T^\infty(\dot{B}^{-s})}+\|u',\mathcal{G}_{-s}u\|_{L_T^1(\dot{B}^{-s})}\le C\|x\|_{\dot{B}^{-s}}+C\|f\|_{L_T^1(\dot{B}^{-s})}.
\end{align*}
\end{theorem}
\begin{proof}
Let us only give the proof of the first part. The homogeneous part $e^{t\mathcal{G}_s}x$ is a classical solution to the homogeneous equation, and satisfies the desired estimate by Lemma \ref{T is a bounded analytic semigroup} and \eqref{maximal L1 regularity criterion positive s}. Denote the inhomogeneous part by $\mathcal{I}f(t)=\int_0^t e^{(t-\tau)\mathcal{G}_s}f(\tau)\,d\tau$. Then it is easy to see that $\|\mathcal{I}f\|_{L_T^\infty(\dot{B}^{s})}\lesssim\|f\|_{L_T^1(\dot{B}^{s})}$. Using again \eqref{maximal L1 regularity criterion positive s} and Fubini's theorem, we have
\begin{align*}
\|\mathcal{G}_s\mathcal{I}f\|_{L_T^1(\dot{B}^{s})}\le&\int_0^T\int_0^t\|\mathcal{G}_se^{(t-\tau)\mathcal{G}_s}f(\tau)\|_{\dot{B}^{s}}\,d\tau\,dt\\
=&\int_0^T\,d\tau\int_{\tau}^T\|\mathcal{G}_se^{(t-\tau)\mathcal{G}_s}f(\tau)\|_{\dot{B}^{s}}\,dt\lesssim\|f\|_{L_T^1(\dot{B}^{s})}.
\end{align*}
So by Lemma \ref{mild to strong}, $u=e^{t\mathcal{G}_s}x+\mathcal{I}f(t)$ is a strong solution to \eqref{ACP positive s}. The estimate for $u'$ follows directly by the previous estimates and the equation \eqref{ACP positive s}. So the proof is completed.
\end{proof}

Now we consider the abstract Cauchy problem \eqref{ACP associated with composite operator} associated with a composite operator of the form
$\mathcal{S}=\mathcal{B}\mathcal{A}$, where $\mathcal{B}$ is a bounded invertible operator and $\mathcal{A}$ is an unbounded operator. Our analysis relies on an intriguing equivalent characterization of norms (see {\cite[Theorem~3.9]{huan 2021}}). Such a result will be very useful for studying density-dependent viscous fluids.

We start with some assumptions.

\begin{assumption}\label{assume A generates a bounded analytic semigroup}
The linear operator $\mathcal{A}:D(\mathcal{A})\subset X\rightarrow X$ generates a bounded analytic semigroup $e^{t\mathcal{A}}$ satisfying $\lim_{t\rightarrow\infty}\|e^{t\mathcal{A}}x\|=0$ for every $x\in X$.
\end{assumption}

\begin{assumption}\label{assume B is bounded and invertible}
$\mathcal{B}\in\mathscr{L}(X)$ is invertible with an inverse $\mathcal{B}^{-1}\in\mathscr{L}(X)$.
\end{assumption}

\begin{assumption}\label{assume BA generates a bounded analytic semigroup}
$\mathcal{S}=\mathcal{B}\mathcal{A}:D(\mathcal{A})\subset X\rightarrow X$ generates a bounded analytic semigroup $e^{t\mathcal{S}}$ satisfying $\lim_{t\rightarrow\infty}\|e^{t\mathcal{S}}x\|=0$ for every $x\in X$.
\end{assumption}

\begin{lemma}\label{intriguing lemma}
Under Assumptions \ref{assume A generates a bounded analytic semigroup}-\ref{assume BA generates a bounded analytic semigroup}, it holds for any $(s,q)\in(0,1)\times[1,\infty]$ and $x\in X$ that
\begin{align}\label{equivalence of norms with negative regularity}
\left\|t^{s}\|e^{t\mathcal{S}}x\|\right\|_{L^q(\R_+,\frac{dt}{t})}\simeq\left\|t^{s}\|e^{t\mathcal{A}}\mathcal{B}^{-1}x\|\right\|_{L^q(\R_+,\frac{dt}{t})}.
\end{align}
Consequently, we have for any $x\in D(\mathcal{A})$,
\begin{align}\label{equivalence of norms with positive regularity}
\left\|t^{-s}\|t\mathcal{S}e^{t\mathcal{S}}x\|\right\|_{L^q(\R_+,\frac{dt}{t})}\simeq\left\|t^{-s}\|t\mathcal{A}e^{t\mathcal{A}}x\|\right\|_{L^q(\R_+,\frac{dt}{t})}.
\end{align}
\begin{proof}
By Assumption \ref{assume A generates a bounded analytic semigroup}, we have for any $x\in X$ that
\begin{align*}
x=-\lim_{\varepsilon\rightarrow0^+}\int_{\varepsilon}^{1/\varepsilon}\mathcal{A}e^{\tau\mathcal{A}}x\,d\tau,
\end{align*}
where the limit converges in $X$. Replacing $x$ by $\mathcal{B}^{-1}x$ gives
\begin{align*}
\mathcal{B}^{-1}x=-\int_0^\infty\mathcal{A}e^{\tau\mathcal{A}}\mathcal{B}^{-1}x\,d\tau.
\end{align*}
Applying $e^{t\mathcal{S}}\mathcal{B}$ to both sides of the above identity, we obtain
\begin{align*}
e^{t\mathcal{S}}x=-\int_0^\infty e^{t\mathcal{S}}\mathcal{B}\mathcal{A}e^{\tau\mathcal{A}}\mathcal{B}^{-1}x\,d\tau.
\end{align*}
We can bound the integrand in two different ways:
\begin{align*}
\|e^{t\mathcal{S}}\mathcal{B}\mathcal{A}e^{\tau\mathcal{A}}\mathcal{B}^{-1}x\|=\|e^{t\mathcal{S}}\mathcal{S}e^{\tau\mathcal{A}}\mathcal{B}^{-1}x\|\lesssim\frac1t\|e^{\tau\mathcal{A}}\mathcal{B}^{-1}x\|\lesssim\frac1t\|e^{\frac{\tau}{2}\mathcal{A}}\mathcal{B}^{-1}x\|,
\end{align*}
or,
\begin{align*}
\|e^{t\mathcal{S}}\mathcal{B}\mathcal{A}e^{\tau\mathcal{A}}\mathcal{B}^{-1}x\|\lesssim\|\mathcal{A}e^{\tau\mathcal{A}}\mathcal{B}^{-1}x\|\lesssim\frac{1}{\tau}\|e^{\frac{\tau}{2}\mathcal{A}}\mathcal{B}^{-1}x\|.
\end{align*}
So we arrive at
\begin{align*}
\|e^{t\mathcal{S}}x\|\lesssim\int_0^\infty\frac{1}{t\vee\tau}\|e^{\tau\mathcal{A}}\mathcal{B}^{-1}x\|\,d\tau.
\end{align*}
Multiplying both sides by $t^s$, we get
\begin{align*}
t^s\|e^{t\mathcal{S}}x\|\lesssim\int_0^\infty\left(\frac{t}{\tau}\right)^s\left(1\wedge\frac{\tau}{t}\right)\tau^s\|e^{\tau\mathcal{A}}\mathcal{B}^{-1}x\|\,\frac{d\tau}{\tau}.
\end{align*}
Since $s\in(0,1)$, it is easy to verify that
\begin{align*}
\sup_{t>0}\int_0^\infty\left(\frac{t}{\tau}\right)^{s}\left(\frac{\tau}{t}\wedge1\right)\,\frac{d\tau}{\tau}+\sup_{\tau>0}\int_0^\infty\left(\frac{t}{\tau}\right)^{s}\left(\frac{\tau}{t}\wedge1\right)\,\frac{dt}{t}
\le C.
\end{align*}
It then follows from {\cite[Lemma~3.7]{huan 2021}} that
\begin{align*}
\left\|t^{s}\|e^{t\mathcal{S}}x\|\right\|_{L^q(\R_+,\frac{dt}{t})}\lesssim\left\|t^{s}\|e^{t\mathcal{A}}\mathcal{B}^{-1}x\|\right\|_{L^q(\R_+,\frac{dt}{t})}.
\end{align*}

The reverse inequality can be proved in a similar way. By Assumption \ref{assume BA generates a bounded analytic semigroup}, we have for any $x\in X$ that
\begin{align*}
x=-\int_0^\infty\mathcal{B}\mathcal{A}e^{\tau\mathcal{S}}x\,d\tau.
\end{align*}
This time we apply $e^{t\mathcal{A}}\mathcal{B}^{-1}$ to both sides of the above identity to get
\begin{align*}
e^{t\mathcal{A}}\mathcal{B}^{-1}x=-\int_0^\infty e^{t\mathcal{A}}\mathcal{A}e^{\tau\mathcal{S}}x\,d\tau.
\end{align*}
So bounding the integrand in two different ways as before gives rise to
\begin{align*}
\|e^{t\mathcal{A}}\mathcal{B}^{-1}x\|\lesssim\int_0^\infty\frac{1}{t\vee\tau}\|e^{\tau\mathcal{S}}x\|\,d\tau.
\end{align*}
This can further imply that
\begin{align*}
\left\|t^{s}\|e^{t\mathcal{A}}\mathcal{B}^{-1}x\|\right\|_{L^q(\R_+,\frac{dt}{t})}\lesssim\left\|t^{s}\|e^{t\mathcal{S}}x\|\right\|_{L^q(\R_+,\frac{dt}{t})}.
\end{align*}
Thus, we have verified \eqref{equivalence of norms with negative regularity}. 

Finally, \eqref{equivalence of norms with positive regularity} follows by replacing $x$ by $\mathcal{S}x$ in \eqref{equivalence of norms with negative regularity}.
\end{proof}
\end{lemma}

We assume additionally that
\begin{assumption}\label{assume A is injective}
$\mathcal{A}:D(\mathcal{A})\subset X\rightarrow X$ is one-to-one.
\end{assumption}
So $\mathcal{S}$ satisfies Assumption \ref{assume S is injective and generates bounded analytic semigroup}.  Then the equivalence of norms implies the equivalence of spaces. More precisely, we get immediately from Lemma \ref{intriguing lemma} that

\begin{corollary}\label{corollary of spaces identification}
Let $s\in(0,2)$. Under Assumptions \ref{assume A generates a bounded analytic semigroup}-\ref{assume A is injective}, we have

{\rm (\romannumeral1)} $\dot{B}_{X,1}^{s,\mathcal{S}}=\dot{B}_{X,1}^{s,\mathcal{A}}$ with equivalent norms,

{\rm (\romannumeral2)} $\dot{B}_{X,1}^{-s,\mathcal{S}}$ coincides with the completion of $R(\mathcal{S})$ with respect to the norm $\|\mathcal{B}^{-1}\cdot\|_{\dot{B}_{X,1}^{-s,\mathcal{A}}}$,

\noindent where the spaces and norms associated with $\mathcal{A}$ are defined in an obvious way.
\end{corollary}

It turns out that the operator $\mathcal{B}$ acting on $\dot{B}_{X,1}^{-s,\mathcal{A}}$ is meaningful. Indeed, \eqref{equivalence of norms with negative regularity} implies that $\mathcal{B}|_{R(\mathcal{A})}$ extends to a continuous operator, denoted by $\overline{\mathcal{B}}$, from $\dot{B}_{X,1}^{-s,\mathcal{A}}$ to $\dot{B}_{X,1}^{-s,\mathcal{S}}$; and that $\mathcal{B}^{-1}|_{R(\mathcal{S})}$ extends to a continuous operator, denoted by $\overline{\mathcal{B}^{-1}}$, from $\dot{B}_{X,1}^{-s,\mathcal{S}}$ to $\dot{B}_{X,1}^{-s,\mathcal{A}}$. Obviously, $\overline{\mathcal{B}}$ is invertible and $\overline{\mathcal{B}}^{-1}=\overline{\mathcal{B}^{-1}}$.
These facts can help us identify  $\mathcal{G}_{-s}$ in the following

\begin{lemma}\label{identify G negative s}
Assuming Assumptions \ref{assume A generates a bounded analytic semigroup}-\ref{assume A is injective}, then the operator
\begin{align*}
\mathcal{A}:D(\mathcal{A})\cap R(\mathcal{S})\subset\dot{B}_{X,1}^{-s,\mathcal{S}}\rightarrow\dot{B}_{X,1}^{-s,\mathcal{A}}
\end{align*}
is closable. Moreover, we have $\mathcal{G}_{-s}=\overline{\mathcal{B}}\,\overline{\mathcal{A}}$, where $\overline{\mathcal{A}}$ is the closure of the above $\mathcal{A}$.
\end{lemma}
\begin{proof}
We see from Lemma \ref{G is a closure of S} {\rm (\romannumeral2)} that $\mathcal{G}_{-s}$ is the closure of
\begin{align*}
\mathcal{B}\mathcal{A}:D(\mathcal{A})\cap R(\mathcal{S})\subset\dot{B}_{X,1}^{-s,\mathcal{S}}\rightarrow\dot{B}_{X,1}^{-s,\mathcal{S}}.
\end{align*}
It follows that $\overline{\mathcal{A}}\vcentcolon=\overline{\mathcal{B}}^{-1}\mathcal{G}_{-s}$ is the closure of
\begin{align*}
\mathcal{A}:D(\mathcal{A})\cap R(\mathcal{S})\subset\dot{B}_{X,1}^{-s,\mathcal{S}}\rightarrow\dot{B}_{X,1}^{-s,\mathcal{A}}.
\end{align*}
This completes the proof.
\end{proof}

We conclude this section with the maximal $L^1$ regularity for the Cauchy problem
\begin{align}\label{ACP associated with closure of A}
\overline{\mathcal{B}}^{-1}u'(t)-\overline{\mathcal{A}}u(t)=f(t),\ \ u(0)=x.
\end{align}

\begin{theorem}\label{Maximal regularity theorem associated with closure of A}
Let $s\in(0,2)$ and $T\in(0,\infty]$. Assuming Assumptions \ref{assume A generates a bounded analytic semigroup}-\ref{assume A is injective}, if $x\in\dot{B}_{X,1}^{-s,\mathcal{S}}$ and $f\in L^1((0,T);\dot{B}_{X,1}^{-s,\mathcal{A}})$, then \eqref{ACP associated with closure of A} has a unique strong solution $u$ in the class
\begin{align*}
u\in C([0,T);\dot{B}_{X,1}^{-s,\mathcal{S}}),\ u'\in L^1((0,T);\dot{B}_{X,1}^{-s,\mathcal{S}}),\ \overline{\mathcal{A}}u\in L^1((0,T);\dot{B}_{X,1}^{-s,\mathcal{A}}).
\end{align*}
Moreover, it holds that
\begin{align*}
\|\overline{\mathcal{B}}^{-1}u\|_{L_T^\infty(\dot{B}_{X,1}^{-s,\mathcal{A}})}+\|\overline{\mathcal{B}}^{-1}u',\overline{\mathcal{A}}u\|_{L_T^1(\dot{B}_{X,1}^{-s,\mathcal{A}})}\le C\|\overline{\mathcal{B}}^{-1}x\|_{\dot{B}_{X,1}^{-s,\mathcal{A}}}+C\|f\|_{L_T^1(\dot{B}_{X,1}^{-s,\mathcal{A}})},
\end{align*}
where $C$ depends on $s$, $\|\mathcal{B}\|_{\mathscr{L}(X)}$ and $\|\mathcal{B}^{-1}\|_{\mathscr{L}(X)}$.
\end{theorem}
\begin{proof}
Note that $\overline{\mathcal{B}}f\in L^1((0,T);\dot{B}_{X,1}^{-s,\mathcal{S}})$. Thanks to the continuity of $\overline{\mathcal{B}}$ and $\overline{\mathcal{B}}^{-1}$, and Lemma \ref{identify G negative s}, then Theorem \ref{Maximal regularity theorem associated with closure of A} follows by applying Theorem \ref{Maximal regularity theorem associated with Gs} {\rm (\romannumeral2)} to the Cauchy problem
\begin{align*}
u'(t)-\overline{\mathcal{B}}\,\overline{\mathcal{A}}u(t)=\overline{\mathcal{B}}f(t),\ \ u(0)=x.
\end{align*}
\end{proof}

\bigskip

\section{Concrete examples}\label{concrete parabolic system section}

\bigskip

In this section, we apply the abstract theory to two concrete examples. The linear system to be considered reads
\begin{eqnarray}\label{concrete Cauchy problem}
\left\{\begin{aligned}
&\rho\partial_t u-\mathcal{A}u=f,\ \ &\mathrm{in}\ (0,\infty)\times\R^n,\\
&u(0)=u_0,\ \ &\mathrm{on}\ \R^n,
\end{aligned}\right.
\end{eqnarray}
where the coefficient $\rho$ is a time-independent function satisfying \eqref{density bounds for Lame system}, and $\mathcal{A}$ is either the Laplacian $\Delta$ or the Lam\'{e} operator $\mathcal{L}$ defined by \eqref{Lame operator}. We denote $b=\rho^{-1}$. From now on, we always assume
\begin{assumption}
$n\ge2$ if $\mathcal{A}=\Delta$, or $n\in\{2,3\}$ if $\mathcal{A}=\mathcal{L}$.
\end{assumption}

We choose $X=L^p=L^p(\R^n;\R^n)$ ($1<p<\infty$), $D(\mathcal{A})=W^{2,p}=W^{2,p}(\R^n;\R^n)$, and $\mathcal{S}=b\mathcal{A}$. Obviously, Assumptions \ref{assume B is bounded and invertible} and \ref{assume A is injective} are satisfied. That $\mathcal{A}$ satisfies Assumption \ref{assume A generates a bounded analytic semigroup} is a classical result (see, e.g., {\cite[Example~3.7.6]{arendt book 2011}}). That $b\Delta:W^{2,p}\subset L^p\rightarrow L^p$ satisfies Assumption \ref{assume BA generates a bounded analytic semigroup} was essentially proved in \cite{mcintosh 2000,duong DIE 1999}. Analogously, we can use Lemma \ref{bL generates bounded analytic semigroup}, Lemma \ref{first step Gaussian upper bound} and Remark \ref{remark on analyticity on Lp} to show that $b\mathcal{L}$ satisfies Assumption \ref{assume BA generates a bounded analytic semigroup} as well.

Let us identify the spaces $\Dot{B}_{X,1}^{\pm s,\mathcal{A}}$. Let $s\in(0,2)$. We know from Lemmas \ref{characterization via classic heat kernel} and \ref{characterization via heat kernel of Lame operator} that the $\Dot{B}_{X,1}^{-s,\mathcal{A}}$-norm is equivalent to the Besov $\Dot{B}_{p,1}^{-s}$-norm. One can see from \eqref{from Laplacian to Lame} and \eqref{from Lame to Laplacian} that $R(\Delta)=R(\mathcal{L})$. It is however easy to see that $R(\Delta)$ is dense in $\Dot{B}_{p,1}^{-s}$. So $\Dot{B}_{X,1}^{-s,\mathcal{A}}$ is identified as $\Dot{B}_{p,1}^{-s}$ for every $s\in(0,2)$. To identify $\Dot{B}_{X,1}^{s,\mathcal{A}}$, we assume additionally $s\le\frac{n}{p}$ so that $\Dot{B}_{p,1}^{s}$ is complete. Then applying Corollary \ref{corollary of spaces identification} {\rm (\romannumeral1)}, Lemmas \ref{characterization via classic heat kernel} and \ref{characterization via heat kernel of Lame operator}, and the obvious fact that $D(\mathcal{A})=W^{2,p}$ is dense in $\dot{B}_{p,1}^{s}$, we get $\dot{B}_{X,1}^{s,\mathcal{S}}=\dot{B}_{X,1}^{s,\mathcal{A}}=\dot{B}_{p,1}^{s}$.

We now turn to the central problem of this section, that is, the maximal $L^1$ regularity for \eqref{concrete Cauchy problem}. In view of Theorem \ref{Maximal regularity theorem associated with Gs} {\rm (\romannumeral1)} and Lemma \ref{G is a closure of S} {\rm (\romannumeral1)}, the smooth solutions to \eqref{concrete Cauchy problem} should satisfy the {\it a priori} estimate
\begin{align*}
\|u\|_{L_T^\infty(\dot{B}_{p,1}^{s})}+\|u',b\mathcal{A} u\|_{L_T^1(\dot{B}_{p,1}^{s})}\lesssim\|u_0\|_{\dot{B}_{p,1}^{s}}+\|bf\|_{L_T^1(\dot{B}_{p,1}^{s})}.
\end{align*}
But if $\rho$ merely satisfies \eqref{density bounds for Lame system}, we can not handle the inhomogeneous term, nor can we obtain the estimate for $\|\mathcal{A}u\|_{L_T^1(\dot{B}_{p,1}^{s})}$. Solving \eqref{concrete Cauchy problem} in Besov spaces with negative regularity seems to be a more promising way to lower the regularity of the density. In fact, from Theorem \ref{Maximal regularity theorem associated with closure of A}, the {\it a priori} estimate for smooth solutions becomes
\begin{align*}
\|\rho u\|_{L_T^\infty(\dot{B}_{p,1}^{-s})}+\|\rho u',\mathcal{A}u\|_{L_T^1(\dot{B}_{p,1}^{-s})}\lesssim\|\rho u_0\|_{\dot{B}_{p,1}^{-s}}+\|f\|_{L_T^1(\dot{B}_{p,1}^{-s})}.
\end{align*}
Unfortunately, the above is not quite true if $u$ is only a strong solution.

By Corollary \ref{corollary of spaces identification} {\rm (\romannumeral2)}, the space $\Dot{B}_{X,1}^{-s,\mathcal{S}}$ agrees with the completion of $(b\mathcal{A}(W^{2,p}),\|\rho\cdot\|_{\dot{B}_{p,1}^{-s}})$, where $b\mathcal{A}(W^{2,p})$ is defined as $\{u=b\mathcal{A}v|v\in W^{2,p}\}$. Then the multiplication by $\rho$ extends to a bounded operator from $\Dot{B}_{X,1}^{-s,\mathcal{S}}$ to $\dot{B}_{p,1}^{-s}$ with a bounded inverse that coincides with the extension of the multiplication by $b$. By Lemma \ref{identify G negative s}, the operator
\begin{align*}
\mathcal{A}: W^{2,p}\cap b\mathcal{A}(W^{2,p})\subset \dot{B}_{X,1}^{-s,\mathcal{S}}\rightarrow \dot{B}_{p,1}^{-s} 
\end{align*}
is closable, and we denote its closure by $\overline{\mathcal{A}}$. Then One can directly interpret Theorem \ref{Maximal regularity theorem associated with closure of A} as follows:
\begin{corollary}\label{corollary of maximal regularity between concrete and abstract}
Let $s\in(0,2)$ and $T\in(0,\infty]$. If $u_0\in \dot{B}_{X,1}^{-s,\mathcal{S}}$ and $f\in L^1((0,T);\dot{B}_{p,1}^{-s})$, then \eqref{concrete Cauchy problem} has a unique strong solution $u$ in the class
\begin{align*}
u\in C([0,T);\dot{B}_{X,1}^{-s,\mathcal{S}}),\  \partial_tu\in L^1((0,T);\dot{B}_{X,1}^{-s,\mathcal{S}}),\  \overline{\mathcal{A}}u\in L^1((0,T);\dot{B}_{p,1}^{-s}).
\end{align*}
Moreover, there exists some constant $C=C(s,m,\mu,\nu)$ such that
\begin{align}\label{advantage of maximal regularity negative s}
\|\rho u\|_{L_T^\infty(\dot{B}_{p,1}^{-s})}+\|\rho u',\overline{\mathcal{A}}u\|_{L_T^1(\dot{B}_{p,1}^{-s})}\le C\|\rho u_0\|_{\dot{B}_{p,1}^{-s}}+C\|f\|_{L_T^1(\dot{B}_{p,1}^{-s})}.
\end{align}
\end{corollary}
Unfortunately, it is not clear whether $\|\nabla u\|_{\infty}$ can be bounded by $\|\overline{\mathcal{A}}u\|_{\dot{B}_{p,1}^{n/p-1}}$ for $n<p<\infty$. Note that an element in $\Dot{B}_{X,1}^{-s,\mathcal{S}}$ might not even be a distribution. So Theorem \ref{Maximal regularity theorem associated with Gs} and Corollary \ref{corollary of maximal regularity between concrete and abstract} may be too abstract to be useful in applications. For this, we require a little more regularity on the coefficients. Recall that $\rho$ is called a multiplier for a function space $(X,\|\cdot\|)$ if $\rho$ defines a continuous linear operator on $X$ by pointwise multiplication. If $\rho$ is a multiplier for $X$, we write $\rho\in\mathscr{M}(X)$ and define the multiplier norm by
\begin{align*}
\|\rho\|_{\mathscr{M}(X)}\vcentcolon=\sup_{\phi\in X}\|\rho\phi\|/\|\phi\|.
\end{align*}

\begin{lemma}\label{finally identify G}
{\rm (\romannumeral1)} Let $p\in(1,\infty)$ and $s\in(0,2)\cap(0,\frac{n}{p}]$. Assume that $\rho,b\in\mathscr{M}(\dot{B}_{p,1}^{s})$. Then $\mathcal{G}_s$ coincides with the operator
\begin{align}\label{bA with domain homogeneous Besov space}
b\mathcal{A}: \dot{B}_{p,1}^{s}\cap \dot{B}_{p,1}^{2+s}\subset\dot{B}_{p,1}^{s}\rightarrow\dot{B}_{p,1}^{s}.
\end{align}

{\rm (\romannumeral2)} Let $p\in(1,\infty)$ and $s\in(0,2)$. Assume that $\rho,b\in\mathscr{M}(\dot{B}_{p,1}^{-s})$. Then the space $\Dot{B}_{X,1}^{-s,\mathcal{S}}$ coincides with $\dot{B}_{p,1}^{-s}$, and the operator $\overline{\mathcal{A}}$ is given by
\begin{align}\label{A with domain homogeneous Besov space}
\mathcal{A}:\dot{B}_{p,1}^{2-s}\cap\dot{B}_{p,1}^{-s}\subset \dot{B}_{p,1}^{-s}\rightarrow\dot{B}_{p,1}^{-s}.
\end{align}
\end{lemma}
\begin{proof}
{\rm (\romannumeral1)} First, along the same lines of the proof of Lemma \ref{G is a closure of S} {\rm (\romannumeral1)}, we can show that $\mathcal{G}_{s}$ is the closure of 
\begin{align}\label{S with domain inhomogeneous Besov space}
\mathcal{S}:\{u\in D(\mathcal{S})|\mathcal{S}u\in\dot{B}_{p,1}^{s}\}\subset\dot{B}_{p,1}^{s}\rightarrow\dot{B}_{p,1}^{s}.   
\end{align}
Since $\rho,b\in\mathscr{M}(\dot{B}_{p,1}^{s})$, we can identify $\{u\in D(\mathcal{S})|\mathcal{S}u\in\dot{B}_{p,1}^{s}\}=\{u\in W^{2,p}|b\mathcal{A}u\in\dot{B}_{p,1}^{s}\}$ as the inhomogeneous Besov space $B_{p,1}^{2+s}=L^p\cap\Dot{B}_{p,1}^{2+s}$. On the other hand, it is easy to see that the operator $b\mathcal{A}$ defined in \eqref{bA with domain homogeneous Besov space} is closed and is an extension of the operator $\mathcal{S}$ defined in \eqref{S with domain inhomogeneous Besov space}. The desired result then follows from the fact that $B_{p,1}^{2+s}$ is dense in $\dot{B}_{p,1}^{s}\cap \dot{B}_{p,1}^{2+s}$.

{\rm (\romannumeral2)} Let us first refine several results in Section \ref{maximal regularity for ACP section}. Using \eqref{equivalence of norms with negative regularity} and the fact that $R(\mathcal{A})=\mathcal{A}(W^{2,p})$ is dense in $\Dot{B}_{p,1}^{-s}$, we can verify that $\Dot{B}_{X,1}^{-s,\mathcal{S}}$ agrees with the completion of 
\begin{align*}
\mathscr{D}_{-s}\vcentcolon=\{u\in L^p|\|\rho u\|_{\Dot{B}_{p,1}^{-s}}<\infty\}  
\end{align*}
with respect to the norm $\|\rho\cdot\|_{\Dot{B}_{p,1}^{-s}}$. Then \eqref{uniform bounded in B negative s} holds for every $u\in \mathscr{D}_{-s}$, so $\mathcal{T}_{-s}(t)$ is the continuous extension of $e^{t\mathcal{S}}|_{\mathscr{D}_{-s}}$ to $\Dot{B}_{X,1}^{-s,\mathcal{S}}$. From this, we can follow the same lines as the proof of Lemma \ref{G is a closure of S} {\rm (\romannumeral2)} to show that $\mathcal{G}_{-s}$ is the closure of 
\begin{align*}
b\mathcal{A}:W^{2,p}\cap\Dot{B}_{X,1}^{-s,\mathcal{S}}\subset\Dot{B}_{X,1}^{-s,\mathcal{S}}\rightarrow\Dot{B}_{X,1}^{-s,\mathcal{S}}.  
\end{align*}

Now assuming $\rho,b\in\mathscr{M}(\dot{B}_{p,1}^{-s})$, it is easy to see that $\Dot{B}_{X,1}^{-s,\mathcal{S}}$ coincides with $\dot{B}_{p,1}^{-s}$. So $\rho\mathcal{G}_{-s}$ is the closure of
\begin{align*}
\mathcal{A}:W^{2,p}\cap\Dot{B}_{p,1}^{-s}\subset\Dot{B}_{p,1}^{-s}\rightarrow\Dot{B}_{p,1}^{-s}.  
\end{align*}
But it is not difficult to see that the closure of the above defined operator is the one defined by \eqref{A with domain homogeneous Besov space}. This completes the proof.
\end{proof}

Finally, we obtain a concrete version of maximal $L^1$ regularity for \eqref{concrete Cauchy problem}.

\begin{theorem}\label{concrete maximal regularity theorem}
Let $p\in(1,\infty)$, $s\in(0,2)$ and $T\in(0,\infty]$. Let $\rho$ satisfy \eqref{density bounds for Lame system} and $b=\rho^{-1}$.

{\rm (\romannumeral1)} Assume that $s\le\frac{n}{p}$ and $\rho,b\in\mathscr{M}(\dot{B}_{p,1}^{s})$. Then for $u_0\in\dot{B}_{p,1}^{s}$ and $f\in L^1((0,T);\dot{B}_{p,1}^{s})$, the equation \eqref{concrete Cauchy problem} has a unique strong solution $u\in C([0,T);\dot{B}_{p,1}^{s})$ satisfying
\begin{align*}
\|u\|_{L_T^\infty(\dot{B}_{p,1}^{s})}+\|\partial_tu,\mathcal{A}u\|_{L_T^1(\dot{B}_{p,1}^{s})}\le C\|u_0\|_{\dot{B}_{p,1}^{s}}+C\|f\|_{L_T^1(\dot{B}_{p,1}^{s})}
\end{align*}
for some constant $C$ depending on $s,m,\mu,\nu,\|\rho\|_{\mathscr{M}(\dot{B}_{p,1}^{s})}$, and $\|b\|_{\mathscr{M}(\dot{B}_{p,1}^{s})}$.

{\rm (\romannumeral2)} Assume $\rho,b\in\mathscr{M}(\dot{B}_{p,1}^{-s})$. If $u_0\in\dot{B}_{p,1}^{-s}$ and $f\in L^1((0,T);\dot{B}_{p,1}^{-s})$, then \eqref{concrete Cauchy problem} has a unique strong solution $u\in C([0,T);\dot{B}_{p,1}^{-s})$ satisfying
\begin{align*}
\|u\|_{L_T^\infty(\dot{B}_{p,1}^{-s})}+\|\partial_tu,\mathcal{A}u\|_{L_T^1(\dot{B}_{p,1}^{-s})}\le C\|u_0\|_{\dot{B}_{p,1}^{-s}}+C\|f\|_{L_T^1(\dot{B}_{p,1}^{-s})}
\end{align*}
for some constant $C$ depending on $s,m,\mu,\nu,\|\rho\|_{\mathscr{M}(\dot{B}_{p,1}^{-s})}$, and $\|b\|_{\mathscr{M}(\dot{B}_{p,1}^{-s})}$.
\end{theorem}
\begin{proof}
The first part follows from Theorem \ref{Maximal regularity theorem associated with Gs} {\rm (\romannumeral1)}, the equivalence between $\dot{B}_{X,1}^{s,\mathcal{S}}$ and $\dot{B}_{p,1}^{s}$, and Lemma \ref{finally identify G} {\rm (\romannumeral1)}. The second part follows from Corollary \ref{corollary of maximal regularity between concrete and abstract} and Lemma \ref{finally identify G} {\rm (\romannumeral2)}.
\end{proof}

\bigskip

\section{An application to pressureless flows}\label{compressible pressureless flow section}

\bigskip

In this section, we study the global-in-time well-posedness for the pressureless flow
\begin{eqnarray}\label{compressible pressureless flow}
\left\{\begin{aligned}
&\partial_t\rho+\divg(\rho u)=0, &\mathrm{in}\ (0,\infty)\times\R^n,\\
&\rho(\partial_tu+u\cdot\nabla u)-\mathcal{L}u=0,\ &\mathrm{in}\ (0,\infty)\times\R^n,\\
&(\rho,u)|_{t=0}=(\rho_0,u_0), &\mathrm{on}\ \R^n,
\end{aligned}\right.
\end{eqnarray}
where $\mathcal{L}$ is the Lam\'{e} operator defined in \eqref{Lame operator} with coefficients satisfying \eqref{Lame coefficients}. The structure of our proof is in the spirit of the one established in \cite{danchin cpam 2012}. But the substantial progress we make is the removal of the smallness assumption on the fluctuation of the initial density.

In this section, we always assume that
\begin{assumption}\label{initial data assumption}
Let $n\in\{2,3\}$, $p\in(1,2n)\setminus\{n\}$, $\rho_0$ satisfy \eqref{density bounds for Lame system}, $u_0\in\Dot{B}_{p,1}^{n/p-1}=(\Dot{B}_{p,1}^{n/p-1}(\R^n))^n$, and $\rho_0,\rho_0^{-1}\in\mathscr{M}(\Dot{B}_{p,1}^{n/p-1})$.
\end{assumption}

Let us be clear about what it means by a solution to the system \eqref{compressible pressureless flow}.
\begin{definition}
The unknown $(\rho,u)$ is called a global-in-time solution to \eqref{compressible pressureless flow} if 
\begin{align*}
\rho\in L^\infty(\R_+\times\R^n)\cap L^\infty(\R_+;\mathscr{M}(\Dot{B}_{p,1}^{n/p-1})),\\
u\in C([0,\infty);\Dot{B}_{p,1}^{n/p-1}),\ (\partial_tu,\mathcal{L}u)\in \left(L^1(\R_+;\Dot{B}_{p,1}^{n/p-1})\right)^2,
\end{align*}
$\rho$ is a weak solution to the continuity equation of \eqref{compressible pressureless flow} (i.e., $\rho$ satisfies \eqref{continuity equation} in the sense of distribution), $(\rho,u)$ satisfies the momentum equation of \eqref{compressible pressureless flow} for a.e. $t\in(0,\infty)$, $u(0)=u_0$, and $\rho(t)\stackrel{\ast}{\rightharpoonup}\rho_0$ in $L^\infty(\R^n)$ as $t\rightarrow0^+$.
\end{definition}

The main result in the section is the following
\begin{theorem}\label{global well posedness of compressible pressureless flow}
Assuming Assumption \ref{initial data assumption}, there exists a positive constant $c$ depending on $m,p,n,\mu,\nu,\|\rho_0\|_{\mathscr{M}(\Dot{B}_{p,1}^{n/p-1})}$ and $\|\rho_0^{-1}\|_{\mathscr{M}(\Dot{B}_{p,1}^{n/p-1})}$ such that if $\|u_0\|_{\Dot{B}_{p,1}^{n/p-1}}\le c$, then \eqref{compressible pressureless flow} has a unique global-in-time solution.
\end{theorem}
\begin{remark}
The above theorem holds without constraint on the dimensions if $\mathcal{L}$ is replaced by $\Delta$.
\end{remark}

Firstly, we shall convert \eqref{compressible pressureless flow} into its Lagrangian formulation. Assume temporarily that $u=u(t,x)$ is a $C^1$ vector field, namely,
\begin{align*}
u\in L_{loc}^1(\R_+;C_b^{1}(\R^n;\R^n)).
\end{align*}
By virtue of Cauchy-Lipschitz theorem, the unique trajectory $X(t,\cdot)$ of $u$, defined by the ODE
\begin{eqnarray}\label{define trajectory via Eulerian velocity}
\left\{\begin{aligned}
&\frac{d}{dt}X(t,y)=u(t,X(t,y)),\\
&X(0,y)=y,
\end{aligned}\right.
\end{eqnarray}
is a $C^1$-diffeomorphism over $\R^n$ for every $t\ge0$. Let us introduce $A(t,y)=\big(D_y X(t,y)\big)^{-1}$, $J(t,y)=\det DX(t,y)$, and $\mathscr{A}(t,y)=\adj DX(t,y)$ (the adjugate of $D X$, i.e.,  $\mathscr{A}=JA$). For any scalar function $\phi=\phi(x)$ and any vector field $v=v(x)$, it is easy to see that
\begin{align}\label{change of variable for gradient}
(\nabla\phi)\circ X=A^{\intercal}\nabla(\phi\circ X),
\end{align}
and
\begin{align}\label{change of variable for divergence1}
(\divg v)\circ X=\mathrm{Tr}[AD(v\circ X)],
\end{align}
where $\mathrm{Tr}$ denotes the trace of a square matrix. On the other hand, using an integration by part argument as in the appendix of \cite{danchin cpam 2012}, we also have
\begin{align}\label{change of variable for divergence2}
(\divg v)\circ X=J^{-1}\divg(\mathscr{A}(v\circ X)).
\end{align}
Applying \eqref{change of variable for gradient} and \eqref{change of variable for divergence1}, we see that
\begin{align}\label{change of variable for gradient divergence}
(\nabla\divg v)\circ X=A^{\intercal}\nabla\mathrm{Tr}(AD(v\circ X)).
\end{align}
By writing $\Delta=\divg\nabla$, we get from \eqref{change of variable for gradient} and \eqref{change of variable for divergence2} that
\begin{align}\label{change of variable for Laplacian}
(\Delta v)\circ X=J^{-1}\divg(\mathscr{A}A^{\intercal}\nabla(v\circ X)).
\end{align}

Now we introduce new unknowns in Lagrangian coordinates and always denote them by bold letters. So, we define
\begin{align}\label{from Eulerian to Lagrangian}
(\lagr{\rho},\lagr{u})(t,y)=(\rho,u)\big(t,X(t,y)\big).
\end{align}
The continuity equation in \eqref{compressible pressureless flow} has a unique weak solution $\rho\in L^\infty(\R_+\times\R^n)$ such that $J\lagr{\rho}\equiv\rho_0$ (see, e.g., {\cite[Proposition~2.1]{ambrosio lecture 2005}}). Using \eqref{change of variable for gradient divergence}, \eqref{change of variable for Laplacian} and the chain rule, one can formally convert the system \eqref{compressible pressureless flow} into its Lagrangian formulation that reads
\begin{eqnarray}\label{Lagrangian formulation of compressible pressureless flow}
\left\{\begin{aligned}
&\rho_0\partial_t \lagr{u}-\mu\divg(\mathscr{A}_{\lagr{u}}A_{\lagr{u}}^{\intercal}\nabla\lagr{u})-(\mu+\lambda)\mathscr{A}_{\lagr{u}}^{\intercal}\nabla\mathrm{Tr}(A_{\lagr{u}}D\lagr{u})=0,\\
&\lagr{u}|_{t=0}=u_0,
\end{aligned}\right.
\end{eqnarray}
where we associate $\mathscr{A}_{\lagr{u}}$ and $A_{\lagr{u}}$ with the new velocity $\lagr{u}$, namely,
\begin{align*}
\mathscr{A}_{\lagr{u}}=\adj DX_{\lagr{u}},\ \ \mathrm{and}\ \ A_{\lagr{u}}=(DX_{\lagr{u}}(t,y))^{-1}
\end{align*}
with
\begin{align}\label{define trajectory via Lagrangian velocity}
X_{\lagr{u}}(t,y)=y+\int_0^t \lagr{u}(\tau,y)\,d\tau.
\end{align}
We shall prove the well-posedness of the highly nonlinear system \eqref{Lagrangian formulation of compressible pressureless flow} using the contraction mapping theorem. Thanks to the linear theory established in Theorem \ref{concrete maximal regularity theorem}, this can be done by rewriting \eqref{Lagrangian formulation of compressible pressureless flow} as
\begin{equation*}
\rho_0\partial_t \lagr{u}-\mathcal{L}\lagr{u}=f(\lagr{u}),  
\end{equation*}
where
\begin{equation*}
f(\lagr{u})=\mu\divg((\mathscr{A}_{\lagr{u}}A_{\lagr{u}}^{\intercal}-I)\nabla\lagr{u})+(\mu+\lambda)\{(\mathscr{A}_{\lagr{u}}^{\intercal}-I)\nabla\mathrm{Tr}(A_{\lagr{u}}D\lagr{u})+\nabla\mathrm{Tr}((A_{\lagr{u}}-I)D\lagr{u})\}. 
\end{equation*}

To bound the nonlinear terms, we need the following

\begin{lemma}[see \cite{danchin cpam 2012, danchin aif 2014}]
Let $\lagr{v}$ be a vector field in $C([0,\infty);\dot{B}_{p,1}^{n/p-1})\cap L^1(\R_+;\dot{B}_{p,1}^{n/p+1})$ and satisfy
\begin{align}\label{smallness on besov norm of v}
\|\nabla\lagr{v}\|_{L^{1}(\dot{B}_{p,1}^{n/p})}\le c_0
\end{align}
for some constant $c_0$. It holds that
\begin{align}\label{flow estimate for A minus I}
\|A_{\lagr{v}}-I\|_{L^\infty(\dot{B}_{p,1}^{n/p})}+\|\mathscr{A}_{\lagr{v}}-I\|_{L^\infty(\dot{B}_{p,1}^{n/p})}\lesssim\|\nabla\lagr{v}\|_{L^1(\dot{B}_{p,1}^{n/p})}.
\end{align}
Let $\lagr{v}_1$ and $\lagr{v}_2$ be two vector fields satisfying the same conditions as $\lagr{v}$, and let $\delta\lagr{v}=\lagr{v}_1-\lagr{v}_2$. Then we have
\begin{align}\label{flow estimate for difference of A}
\|A_{\lagr{v}_1}-A_{\lagr{v}_2}\|_{L^\infty(\dot{B}_{p,1}^{n/p})}+\|\mathscr{A}_{\lagr{v}_1}-\mathscr{A}_{\lagr{v}_2}\|_{L^\infty(\dot{B}_{p,1}^{n/p})}\lesssim\|\nabla\delta\lagr{v}\|_{L^1(\dot{B}_{p,1}^{n/p})}.
\end{align}
\end{lemma}

Now, in view of \eqref{flow estimate for A minus I} and product laws in Besov spaces, we have
\begin{align}\label{nonlinear estimate}
\|f(\lagr{v})\|_{L^{1}(\dot{B}_{p,1}^{n/p-1})}\lesssim\|\nabla\lagr{v}\|_{L^{1}(\dot{B}_{p,1}^{n/p})}^2
\end{align}
whenever $\lagr{v}$ satisfies \eqref{smallness on besov norm of v}.

Again, in view of Theorem \ref{concrete maximal regularity theorem}, we shall perform the contraction mapping theorem in the Banach space $E_p$ defined as
\begin{align*}
E_p\vcentcolon=\left\{\lagr{u}\in C_b([0,\infty);\Dot{B}_{p,1}^{n/p-1})|\partial_t\lagr{u}\in L^1(\R_+;\Dot{B}_{p,1}^{n/p-1}),\lagr{u}\in L^1(\R_+;\Dot{B}_{p,1}^{n/p+1})\right\}
\end{align*}
endowed with the norm
\begin{align*}
\|\lagr{u}\|_{E_p}\vcentcolon=\|\lagr{u}\|_{L^\infty(\Dot{B}_{p,1}^{n/p-1})}+\|\partial_t\lagr{u},\mathcal{L}\lagr{u}\|_{L^1(\Dot{B}_{p,1}^{n/p-1})}.
\end{align*}

Now we can prove the global-in-time well-posedness for \eqref{Lagrangian formulation of compressible pressureless flow}.
\begin{theorem}\label{global well posedness of compressible pressureless flow in Lagrangian coordinates}
Assuming Assumption \ref{initial data assumption}, there exists a positive constant $c$ depending on $m,p,n,\mu,\nu,\|\rho_0\|_{\mathscr{M}(\Dot{B}_{p,1}^{n/p-1})}$ and $\|\rho_0^{-1}\|_{\mathscr{M}(\Dot{B}_{p,1}^{n/p-1})}$ such that if $\|u_0\|_{\Dot{B}_{p,1}^{n/p-1}}\le c$, then \eqref{Lagrangian formulation of compressible pressureless flow} has a unique global-in-time strong solution $\lagr{u}\in E_p$ satisfying $\|\lagr{u}\|_{E_p}\lesssim\|u_0\|_{\Dot{B}_{p,1}^{n/p-1}}$.
\end{theorem}
\begin{proof}
For $r>0$, let $E_p(r)$ denote the closed ball in $E_p$ centered at $u=0$ with radius $r$. We shall construct a contraction mapping on $E_p(r)$ by solving the linearized system
\begin{eqnarray}\label{linearized Lagrangian formulation of pressureless flow}
\left\{\begin{aligned}
&\rho_0\partial_t \lagr{u}-\mathcal{L}\lagr{u}=f(\lagr{v}),\\
&\lagr{u}|_{t=0}=u_0,
\end{aligned}\right.
\end{eqnarray}
where the input $\lagr{v}\in E_p(r)$. To bound the inhomogeneous term, we require $r$ to be small so that
\begin{align*}
\|\nabla\lagr{v}\|_{L^1(\Dot{B}_{p,1}^{n/p})}\le C\|\mathcal{L}\lagr{v}\|_{L^1(\Dot{B}_{p,1}^{n/p-1})}\le C_1r\le c_0.
\end{align*}
This then implies \eqref{nonlinear estimate}. 

Now, applying Theorem \ref{concrete maximal regularity theorem}, we can solve \eqref{linearized Lagrangian formulation of pressureless flow} for a strong solution $\lagr{u}\in E_p$ satisfying
\begin{align*}
\|\lagr{u}\|_{E_p}\le C\|u_0\|_{\dot{B}_{p,1}^{n/p-1}}+C\|f(\lagr{v})\|_{L^1(\dot{B}_{p,1}^{n/p-1})}\le C_2\|u_0\|_{\dot{B}_{p,1}^{n/p-1}}+C_2r^2.
\end{align*}
To ensure that the mapping $\lagr{v}\mapsto\lagr{u}$ is a self-map on $E_p(r)$, we need 
\begin{align*}
r\le \frac{c_0}{C_1}\wedge\frac{1}{2C_2}
\end{align*}
and
\begin{align*}
\|u_0\|_{\dot{B}_{p,1}^{n/p-1}}\le\frac{r}{2C_2}.
\end{align*}

Next, we need to show the contraction property of the mapping $\lagr{v}\mapsto\lagr{u}$. Given $\lagr{v}_1,\lagr{v}_2\in E_p(r)$, let $\lagr{u}_1,\lagr{u}_2\in E_p(r)$ be the corresponding solutions to \eqref{linearized Lagrangian formulation of pressureless flow}. In what follows, for two quantities $q_1$ and $q_2$, we always denote by $\delta q$ their difference $q_1-q_2$. Then applying Theorem \ref{concrete maximal regularity theorem} to the equation satisfied by $\delta\lagr{u}$, we obtain
\begin{align*}
\|\delta\lagr{u}\|_{E_p}\le C\|f(\lagr{v}_1)-f(\lagr{v}_2)\|_{L^1(\dot{B}_{p,1}^{n/p-1})}.
\end{align*}
We write
\begin{align*}
f(\lagr{v}_1)-f(\lagr{v}_2)=&\mu\divg((\mathscr{A}_{1}A_{1}^{\intercal}-I)\nabla\delta\lagr{v})+\mu\divg((\mathscr{A}_{1}A_{1}^{\intercal}-\mathscr{A}_{2}A_{2}^{\intercal})\nabla\lagr{v}_2)\\
&+(\mu+\lambda)(\mathscr{A}_{1}^{\intercal}-I)\nabla\mathrm{Tr}(A_{1}D\delta\lagr{v})+(\mu+\lambda)(\mathscr{A}_{1}^{\intercal}-I)\nabla\mathrm{Tr}(\delta AD\lagr{v}_2)\\
&+(\mu+\lambda)(\delta\mathscr{A})^{\intercal}\nabla\mathrm{Tr}(A_{2}D\lagr{v}_2)+(\mu+\lambda)\nabla\mathrm{Tr}((A_{1}-I)D\delta\lagr{v})\\
&+(\mu+\lambda)\nabla\mathrm{Tr}(\delta A D\lagr{v}_2),
\end{align*}
where $\mathscr{A}_i=\mathscr{A}_{\lagr{v}_i}$ and $A_i=A_{\lagr{v}_i}$, $i=1,2$. Applying \eqref{flow estimate for A minus I}, \eqref{flow estimate for difference of A} and product laws in Besov spaces, we arrive at
\begin{align*}
\|f(\lagr{v}_1)-f(\lagr{v}_2)\|_{L^1(\dot{B}_{p,1}^{n/p-1})}\le C\|\nabla\lagr{v}_1,\nabla\lagr{v}_2\|_{L^1(\dot{B}_{p,1}^{n/p})}\|\nabla\delta\lagr{v}\|_{L^1(\dot{B}_{p,1}^{n/p})}.
\end{align*}
We thus infer
\begin{align*}
\|\delta\lagr{u}\|_{E_p}\le C\|\nabla\lagr{v}_1,\nabla\lagr{v}_2\|_{L^1(\dot{B}_{p,1}^{n/p})}\|\nabla\delta\lagr{v}\|_{L^1(\dot{B}_{p,1}^{n/p})}\le C_3r\|\delta\lagr{v}\|_{E_p},
\end{align*}
from which we see that $\|\delta\lagr{u}\|_{E_p}\le \frac12\|\delta\lagr{v}\|_{E_p}$ if $r\le\frac{1}{2C_3}$.

Finally, we choose
\begin{align*}
r=\frac{c_0}{C_1}\wedge\frac{1}{2C_2}\wedge\frac{1}{2C_3}\ \ \mathrm{and}\ \ c=\frac{r}{2C_2}.
\end{align*}
Then the mapping $\lagr{v}\mapsto\lagr{u}$ is a contraction on $E_p(r)$, thus, admits a unique fixed point $\lagr{u}\in E_p(r)$, which is a solution to \eqref{Lagrangian formulation of compressible pressureless flow} in $E_p$. The proof of the uniqueness of strong solutions in $E_p$ is similar to the proof of the contraction property of the mapping $\lagr{v}\mapsto\lagr{u}$. This completes the proof of the theorem.
\end{proof}
\begin{remark}
For $n<p<2n$, in view of \eqref{advantage of maximal regularity negative s}, one can prove the global well-posedness under the assumption that $\|\rho_0u_0\|_{\Dot{B}_{p,1}^{n/p-1}}\le c$ with $c$ only depending on $m,p,n,\mu,\nu$.
\end{remark}

Let us conclude this paper by proving Theorem \ref{global well posedness of compressible pressureless flow}.
\begin{proof}[Proof of Theorem \ref{global well posedness of compressible pressureless flow}]
Let $\lagr{u}$ be the global-in-time solution to \eqref{Lagrangian formulation of compressible pressureless flow} constructed in Theorem \ref{global well posedness of compressible pressureless flow in Lagrangian coordinates}. Then \eqref{define trajectory via Lagrangian velocity} defines a $C^1$-diffeomorphism $X_{\lagr{u}}(t,\cdot)$ over $\R^n$ for every $t\ge0$, which enables us to go back to the Eulerian coordinates by introducing
\begin{align*}
\rho(t,x)=\rho_0(X_{\lagr{u}}^{-1}(t,x))\ \ \mathrm{and}\ \ u(t,x)=\lagr{u}(t,X_{\lagr{u}}^{-1}(t,x)).
\end{align*}
Then $(\rho,u)$ is a solution to \eqref{compressible pressureless flow}. 

Let $(\rho_i,u_i)$, $i=1,2$, be two solutions to \eqref{compressible pressureless flow} with the same initial data. Let $X_{u_i}$ be defined via \eqref{define trajectory via Eulerian velocity}, and $(\lagr{\rho}_i,\lagr{u}_i)$ via \eqref{from Eulerian to Lagrangian}. Then $\lagr{u}_1$ and $\lagr{u}_2$ are two solutions to \eqref{Lagrangian formulation of compressible pressureless flow} with the same initial data. So it follows from the uniqueness part of Theorem \ref{global well posedness of compressible pressureless flow in Lagrangian coordinates} that $(\rho_1,u_1)=(\rho_2,u_2)$.

We refer the reader to \cite{danchin cpam 2012,danchin aif 2014} for more details.
\end{proof}

\bigskip


\end{document}